\documentclass[11pt,letterpaper]{amsart}

\usepackage{url}

\usepackage{hyperref}

\usepackage{mathtools}

\input{macros.sty}

\newif\iffinalrun

\iffinalrun
\else

\fi

\iffinalrun
        \newcommand{\need}[1]{}
        \newcommand{\mar}[1]{}
\else
        \newcommand{\need}[1]{{ \tiny *** #1}}
        \newcommand{\mar}[1]{\marginpar{\raggedright\tiny #1}}
\fi

\begin{document}

\title{Cohomology of $(\varphi,\Gamma)$-modules over pseudorigid spaces}
\author{Rebecca Bellovin}

\begin{abstract}
	We study the cohomology of families of $(\varphi,\Gamma)$-modules with coefficients in pseudoaffinoid algebras.  We prove that they have finite cohomology, and we deduce an Euler characteristic formula and Tate local duality.  We classify rank-$1$ $(\varphi,\Gamma)$-modules and deduce that triangulations of pseudorigid families of $(\varphi,\Gamma)$-modules can be interpolated, extending a result of ~\cite{kpx}.  We then apply this to study extended eigenvarieties at the boundary of weight space, proving in particular that the eigencurve is proper at the boundary and that Galois representations attached to certain characteristic $p$ points are trianguline.
\end{abstract}

\maketitle

\section{Introduction}

In our earlier paper ~\cite{bellovin2020}, we began studying families of Galois representations varying over \emph{pseudorigid} spaces, that is, families of Galois representations where the coefficients have a non-archimedean topology but which (in contrast to the rigid analytic spaces of Tate) are not required to contain a field.  Such coefficients arise naturally in the study of eigenvarieties at the boundary of weight space.

The theory of $(\varphi,\Gamma)$-modules is a crucial tool in the study of $p$-adic Galois representations.  At the expense of making the coefficients more complicated, it lets us turn the data of a Galois representation into the data of a Frobenius operator $\varphi$ and a $1$-dimensional $p$-adic Lie group $\Gamma$.  Moreover, Galois representations which are irreducible often become reducible on the level of their associated $(\varphi,\Gamma)$-modules.  Such $(\varphi,\Gamma)$-modules have played an important role in the $p$-adic Langlands program.

In our previous paper ~\cite{bellovin2020}, we constructed $(\varphi,\Gamma)$-modules associated to Galois representations varying over pseudorigid spaces. In the present paper, we turn to the study of the cohomology of $(\varphi,\Gamma)$-modules over pseudorigid spaces have finite cohomology, whether or not they come from Galois representations.  Given a $(\varphi,\Gamma)$-module $D$, the Fontaine--Herr--Liu complex $C_{\varphi,\Gamma}^\bullet(D)$ is an explicit three-term complex which, when $D$ arises from a Galois representation, computes the Galois cohomology.  We begin by proving that such families of $(\varphi,\Gamma)$-modules have finite cohomology, extending the main result of ~\cite{kpx}:  
\begin{thm}
	Suppose $D$ is a projective $(\varphi,\Gamma)$-module over a pseudoaffinoid algebra $R$.  Then $C_{\varphi,\Gamma}^\bullet(D)\in D_{\mathrm{perf}}^{[0,2]}(R)$.
\end{thm}

As a corollary, we deduce the Euler characteristic formula:
\begin{cor}
	If $D$ is a projective $(\varphi,\Gamma_K)$-module with coefficients in a pseudoaffinoid algebra $R$, then $\chi(D)=-(\rk D)[K:\Q_p]$.
\end{cor}
This extends the result ~\cite[Theorem 4.3]{liu2007}.  However, the method of proof is different: Liu proved finiteness of cohomology and the Euler characteristic formula at the same time, making a close study of $t$-torsion $(\varphi,\Gamma)$-modules to shift weights.  There is no element $t$ in our setting, because $p$ is not necessarily invertible.  However, because we proved finiteness of cohomology for pseudorigid families of $(\varphi,\Gamma)$-modules first, we can deduce the Euler characteristic formula by deformation, without studying torsion objects.

We then turn to $(\varphi,\Gamma)$-modules with coefficients in finite extensions of $\F_p(\!(u)\!)$, and we prove Tate local duality:
\begin{thm}
Tate local duality holds for every projective $(\varphi,\Gamma)$-module $D$ over $\Lambda_{R,\rig,K}$.
\end{thm}
Here $\Lambda_{R,\rig,K}$ is a mixed- or postitive-characteristic analogue of the usual Robba ring, which we will define in \textsection~\ref{subsec: rings}.  Our proof closely follows that of ~\cite[Theorem 4.7]{liu2007}; we compute the cohomology of many rank-$1$ $(\varphi,\Gamma)$-modules and then proceed by induction on the degree, using the Euler characteristic formula to produce non-split extensions.  We are then able to finish the computation of the cohomology of $(\varphi,\Gamma)$-modules of character type.

With this in hand, we are able to show that all rank-$1$ $(\varphi,\Gamma)$-modules over pseudorigid spaces are of character type, following ~\cite{kpx}, and we deduce that triangulations can be interpolated from a dense set of maximal points (in the sense of ~\cite[Definition 2.2.7]{johansson-newton17}):
\begin{thm}
	Let $X$ be a reduced pseudorigid space, let $D$ be a projective $(\varphi,\Gamma_K)$-module over $X$ of rank $d$, and let $\delta_1,\ldots,\delta_d\colon K^\times\rightarrow\Gamma(X,\mathscr{O}_X^\times)$ be a set of continuous characters.  Suppose there is a very Zariski-dense set $X^{\mathrm{cl}}\subset X$ of maximal points such that for every $x\in X^{\mathrm{cl}}$, $D_x$ is trianguline with parameters $\delta_{1,x},\ldots,\delta_{d,x}$.  Then there exists a proper birational morphism $f\colon X'\rightarrow X$ of reduced pseudorigid spaces and an open subspace $U\subset X'$ containing $\{p=0\}$ such that $f^\ast D|_U$ is trianguline with parameters $f^\ast\delta_1,\ldots,f^\ast\delta_d$.
\end{thm}
Unlike the situation in characteristic $0$, the triangulation extends over every point of characteristic $p$, and there are no critical points.  This is again because there is no analogue of Fontaine's element $t$ in our positive characteristic analogue of the Robba ring.

Finally, we turn to applications to the extended eigenvarieties constructed in ~\cite{johansson-newton}.  Adapting the Galois-theoretic argument of ~\cite{diao-liu}, we prove unconditionally that each irreducible component of the extended eigencurve is proper at the boundary of weight space, and that the Galois representations over characteristic $p$ points of the extended eigencurve are trianguline at $p$.  The latter answers a question of ~\cite{aip2018}.

We actually prove these results under somewhat abstracted hypotheses, in order to facilitate deducing analogous results for other extended eigenvarieties.  In particular, our results apply to certain unitary and Hilbert eigenvarieties.  However, for most groups the necessary results have not been proven even for Galois representations attached to classical forms, nor have the required families of Galois representations been constructed.

In the appendices, we have collected several results on the geometry of pseudorigid spaces and Galois determinants over pseudorigid spaces necessary for our applications.

\begin{remark}
	We assume throughout that $p\neq 2$.  It should be possible to remove this hypothesis without any real difficulty, but we would have had to work systematically with $R\Gamma\left(\Gamma_K^{p\mathrm{-tors}}, C_{\varphi,\Gamma}^\bullet(D) \right)$, rather than the usual Fontaine--Herr--Liu complex.
\end{remark}

\subsection*{Acknowledgements}
I am grateful to Toby Gee, James Newton, and Lynnelle Ye for helpful conversations.

\section{Background}

\subsection{Rings of $p$-adic Hodge theory}
\label{subsec: rings}

Let $R$ be a pseudoaffinoid $\mathscr{O}_E$-algebra, for some finite extension $E/\Q_p$ (we provide a precise definition of pseudoaffinoid algebrais in ~\ref{app: pseudorigid}) with uniformizer $\varpi_E$, with ring of definition $R_0\subset R^\circ$ and pseudouniformizer $u\in R_0$, and assume that $p\not\in R^\times$.  Let $K/\Q_p$ be a finite extension, let $\chi_{\cyc}\colon\Gal_K\rightarrow\Z_p^\times$ be the cyclotomic character, let $H_K\coloneqq\ker \chi_{\cyc}$, and let $\Gamma_K\coloneqq\Gal_K/H_K$.  Given an interval $I\subset [0,\infty]$, we defined rings $(\widetilde\Lambda_{R_0,I,K},\widetilde\Lambda_{R_0,I,K}^+)$ and $(\Lambda_{R_0,I,K},\Lambda_{R_0,I,K}^+)$ in ~\cite[Definition 3.2]{bellovin2020} and ~\cite[Definition 3.40]{bellovin2020}, respectively, which (when $I=[0,b]$) are analogues of the characteristic $0$ rings $(\widetilde{\A}_K^{(0,b]},\widetilde{\A}_K^{\dagger,s(b)})$ and $(\A_K^{(0,r]}, \A_K^{\dagger,s(r)})$ defined in ~\cite{colmez2008}.  Here $s\colon(0,\infty)\rightarrow (0,\infty)$ is defined via $s(r)\coloneqq\frac{p-1}{pr}$.  We briefly recall their definitions here and state some of their properties.

Let $\Ainf\coloneqq W(\O_{{\C}_K}^\flat)$, where $\O_{\C_K}^\flat\coloneqq \varprojlim_{x\mapsto x^p}\O_{\C_K}$ is the tilt of $\O_{\C_K}$.  Let $\varepsilon\coloneqq (\varepsilon^{(0)},\varepsilon^{(1)},\ldots)\in \O_{\C_K}^\flat$ be a choice of a compatible sequence of $p$-power roots of unity, with $\varepsilon^{(0)}=1$ and $\varepsilon^{(1)}\neq1$, and let $\overline\pi\coloneqq \varepsilon -1$ and $\pi\coloneqq [\varepsilon]-1\in\Ainf$.  Then if $I=[a,b]$ for rational numbers $a, b$ with $0\leq a\leq b\leq \infty$, we define $(\widetilde\Lambda_{R_0,I},\widetilde\Lambda_{R_0,I}^+)$ such that
\[	\Spa(\widetilde\Lambda_{R_0,I},\widetilde\Lambda_{R_0,I}^+) = \left(\Spa(R_0\htimes\Ainf,R_0\htimes\Ainf)\right)\left\langle \frac{[\overline\pi]^{s(a)}}{u}, \frac{u}{[\overline\pi]^{s(b)}}\right\rangle	\]
If $a=0$, we take $\frac{[\overline\pi]^{\infty}}{u}=0$, and if $b=0$, we take $\frac{u}{[\overline\pi]^\infty}=\frac{1}{[\overline\pi]}$.

The ring $\widetilde\Lambda_{R_0,I}$ has ring of definition $\left(R_0\htimes\A_{\mathrm{inf}}\right)\left\langle \frac{[\overline\pi]^{s(a)}}{u}, \frac{u}{[\overline\pi]^{s(b)}}\right\rangle$; when $b\neq \infty$, this permits us to define a valuation
\[	 v_{R,[a,b]}(x)\coloneqq \sup_{\alpha\in\C_p^\flat:[\alpha]x\in \widetilde\Lambda_{R_0,[a,b],0}}-v_{\C_p^\flat}(\alpha)  \]
on it.  When $a=0$, we abbreviate $v_{R,[0,b]}$ as $v_{R,b}$.  When $b=\infty$, we let $v_{R,[a,\infty]}$ be the $u$-adic valuation.

The group $H_K$ acts on $(\widetilde\Lambda_{R_0,I},\widetilde\Lambda_{R_0,I}^+)$, because $\Gal_K$ acts on $\Ainf$ and $H_K$ fixes $[\overline\pi]$.  Then by ~\cite[Corollary 3.36]{bellovin2020},
\[	\Spa(\widetilde\Lambda_{R_0,I}^{H_K},\widetilde\Lambda_{R_0,I}^{+,H_K}) = \left(\Spa(R_0\htimes\Ainf^{H_K},R_0\htimes\Ainf^{H_K})\right)\left\langle \frac{[\overline\pi]^{s(a)}}{u}, \frac{u}{[\overline\pi]^{s(b)}}\right\rangle      \]
If $I\subset I'$, we have injective maps $\widetilde\Lambda_{R_0,I'}\rightarrow \widetilde\Lambda_{R_0,I}$ and $\widetilde\Lambda_{R_0,I'}^{H_K}\rightarrow\widetilde\Lambda_{R_0,I}^{H_K}$.  Thus, if $I'$ is an interval with an open endpoint, we may define 
\[	(\widetilde\Lambda_{R_0,I'},\widetilde\Lambda_{R_0,I'}^+)\coloneqq \cap_{\substack{I\subset I'\\ \mathrm{closed}}}(\widetilde\Lambda_{R_0,I},\widetilde\Lambda_{R_0,I}^+)	\]
and 
\[      (\widetilde\Lambda_{R_0,I'}^{H_K},\widetilde\Lambda_{R_0,I'}^{+,H_K})\coloneqq \cap_{\substack{I\subset I'\\ \mathrm{closed}}}(\widetilde\Lambda_{R_0,I}^{H_K},\widetilde\Lambda_{R_0,I}^{+,H_K})       \]

The rings $(\Lambda_{R_0,I,K},\Lambda_{R_0,I,K}^+)$ are imperfect versions of these, defined when $I\subset [0,b]$ with $b$ sufficiently small.  Given $\lambda=\frac{m'}{m}\in\Q_{>0}$ with $\gcd(m,m')=1$, let $(D_\lambda, D_{\lambda}^\circ)$ denote the pair of rings corresponding to the localization $(\mathscr{O}_E[\![u]\!],\mathscr{O}_E[\![u]\!])\left\langle\frac{\varpi_E^m}{u^{m'}}\right\rangle$. By ~\cite[Lemma 4.8]{lourenco}, there is some sufficiently small $\lambda$ such that $R$ is topologically of finite type over $D_\lambda$, so we may assume that $R_0$ is topologically of finite type over $D_\lambda^\circ$, i.e., there is a continuous, open, and surjective homomorphism $D_\lambda\left\langle T_1,\ldots,T_n\right\rangle\twoheadrightarrow R$~\footnote{This is defined as ``strictly topologically of finite type'' in ~\cite{wedhorn}; the definition of ``topologically of finite type'' given there is slightly more general, following ~\cite[\textsection 3]{huber1994}.  But in the case of Tate rings, the two definitions coincide by ~\cite[Lemma 3.5]{huber1994}.}.  

For any unramified extension $F/\Q_p$, the choice of $\varepsilon$ gives us a natural map $k_F(\!(\overline\pi)\!)\rightarrow \C_K^\flat$; let $\E_F$ denote its image, and let $\E\subset \C_K^\flat$ be its separable closure.  Then $\Gal(\E/\E_F)\cong H_F$ (by the theory of the field of norms), and for any extension $K/F$, we set $\E_K\coloneqq \E^{H_K}$.  Then $\E_K$ is a discretely valued field, and we may choose a uniformizer $\overline\pi_K$; if we lift its minimal polynomial to characteristic $0$, Hensel's lemma implies that we have a lift $\pi_K\in W(\C_K^\flat)$ which is integral over $\O_F[\![\pi]\!][\frac 1 \pi]^\wedge$.  We fix a choice $\pi_K$ for each $K$, and work with it throughout (when $F/\Q_p$ is unramified, we take $\pi_F$ to be $\pi$).

Assume that $0\leq a\leq b< r_K\cdot\lambda$, where $r_K$ is a constant defined in ~\cite{colmez2008}, and that $\frac{1}{a\cdot v_{\C_K^\flat}(\overline\pi_K))},\frac{1}{b\cdot v_{\C_K^\flat}(\overline\pi_K))}\in\Z$. Let $F'\subset K_\infty\coloneqq K(\mu_{p^\infty})$ be the maximal unramified subfield.  Then we define $\Lambda_{R_0,[a,b],K}$ to be the evaluation of $\mathscr{O}_{(R_0\otimes\mathscr{O}_{F'})[\![\pi_K]\!]}$ on the affinoid subspace of $\Spa(R_0\otimes\mathscr{O}_{F'})[\![\pi_K]\!]$ defined by the conditions $u\leq \pi_K^{1/(b\cdot v_{\C_K^\flat}(\overline\pi_K))}$ and $\pi_K^{1/(a\cdot v_{\C_K^\flat}(\overline\pi_K))}\leq u$ (and similarly for $\Lambda_{R_0,[a,b],K}^+$).  We further set $\Lambda_{R,[a,b],K}\coloneqq \Lambda_{R_0,[a,b],K}\left[\frac 1 u\right]$.

If $p=0$ in $R$, then we may take $\lambda$ arbitrarily large, and hence $b$ arbitrarily large.  Thus, in this case we additionally define $\Lambda_{R_0,[a,\infty],K}\coloneqq (R_0\otimes_{\Z_p}\mathscr{O}_{F'})[\![\pi_K]\!]$.

We further define $\Lambda_{R,(0,b],K}\coloneqq \varprojlim_{a\rightarrow 0}\Lambda_{R,[a,b],K}$, and $\Lambda_{R,\rig,K}\coloneqq \varinjlim_{b\rightarrow 0}\Lambda_{R,(0,b],K}$.

The rings $\widetilde\Lambda_{R_0,I}^{H_K}$ and $\Lambda_{R_0,I,K}$ are equipped with actions of Frobenius and $\Gamma_K$.  More precisely, we have isomorphisms
\[	\varphi\colon\widetilde\Lambda_{R_0,I}\xrightarrow{\sim}\widetilde\Lambda_{R_0,\frac 1 pI}, \quad \varphi\colon\widetilde\Lambda_{R_0,I}^{H_K}\xrightarrow{\sim}\widetilde\Lambda_{R_0,\frac 1 pI}^{H_K}	\]
and ring homomorphisms
\[	\varphi\colon\Lambda_{R_0,[a,b],K}\rightarrow \Lambda_{R_0,[a/p,b/p],K}	\]
However, the latter are not isomorphisms; $\varphi$ makes $\Lambda_{R_0,[0,b/p],K}$ into a free $\varphi(\Lambda_{R_0,[0,b],K})$-module, with basis $\{1,[\varepsilon],\ldots,[\varepsilon]^{p-1}\}$.  We may define a left inverse $\psi\colon\Lambda_{R_0,[0,b/p],K}\rightarrow \Lambda_{R_0,[0,b],K}$ by defining
\[	\psi\left(\varphi(a_0) + \varphi(a_1)[\varepsilon]+\cdots+\varphi(a_{p-1})[\varepsilon]^{p-1}\right) = a_0	\]
If $p$ is a non-zero-divisor in $R_0$, we may instead write $\psi = p^{-1}\varphi^{-1}\circ\Tr_{\Lambda_{R_0,[0,b/p],K}/\varphi(\Lambda_{R_0,[0,b],K})}$.

There is a natural map $\Lambda_{R_0,[a,b],K}\rightarrow \widetilde\Lambda_{R_0,[a,b],K}$, and so $\Lambda_{R_0,[a,b],K}$ inherits the valuation $v_{R,[a,b]}$.  We can compute $v_{R,b}$ explicitly when $R$ is a finite extension of $\F_p(\!(u)\!)$:
\begin{lemma}
	If $R$ is a a finite extension of $\F_p(\!(u)\!)$, equipped with the $u$-adic valuation $v_R$ (with $v_R(u)=1$), then 
	\[	v_{R,b}\left(\sum_{i\in\Z}a_i\pi_K^i\right) = \frac 1 b \inf_{i\in\Z}\left\{v_R(a_i) + ibv_{\C_p^\flat}(\overline\pi_K)\right\}	\]
	\label{lemma: val 0 b R a field}
\end{lemma}
\begin{proof}
	It is straightforward to check that $v_{R,b}(a_i\pi_K^i) = iv_{\C_p^\flat}(\overline\pi_K) + \frac{v_R(a_i)}{b}$, which yields the claim.
\end{proof}

We can also estimate the $u$-adic valuation $v_{R,[a,\infty]}$:
\begin{lemma}
	If $R$ is a a finite extension of $\F_p(\!(u)\!)$, equipped with the $u$-adic valuation $v_R$ (with $v_R(u)=1$), then
	\[      \inf_{i\geq 0}\left\{v_R(a_i) + iav_{\C_p^\flat}(\overline\pi_K)\right\}	\]
	is a valuation on $\Lambda_{R,[a,\infty],K}$ whose ring of integers is $\Lambda_{R,[a,\infty],K,0}$.
	\label{lemma: val a infty R a field}
\end{lemma}
\begin{proof}
	We again compute the valuation of monomials: $v_{R,[a,\infty]}(a_i\pi_K^i) = v_R(a_i) + \lfloor iav_{\C_p^\flat}(\overline\pi_K)\rfloor$.
\end{proof}

Thus, we may define an auxiliary valuation $v_{R,a}'$ on $\Lambda_{R,[a,\infty],K}$ via
\[	v_{R,a}'\left(\sum_{i\geq 0}a_i\pi_K^i\right) \coloneqq  \inf_{i\geq 0}\left\{v_R(a_i) + iav_{\C_p^\flat}(\overline\pi_K)\right\}        \]

By ~\cite[Proposition 3.10]{bellovin2020}, the formation of $\widetilde\Lambda_{R,I}$ behaves well with respect to rational localization on $\Spa R$, and $\Lambda_{R_0,I,K}$ does, as well, since it is sheafy.  Thus, if $X$ is a (not necessarily affinoid) pseudorigid space, we may let $\widetilde\Lambda_{X,I}^{H_K}$ and $\Lambda_{X,I,K}$ denote the corresponding sheaves of algebras on $X$.

\subsection{\texorpdfstring{$(\varphi,\Gamma)$}{(𝜑, Γ)}-modules and cohomology}

We briefly recall the theory of $(\varphi,\Gamma)$-modules over pseudorigid spaces.

\begin{definition}
	A $\varphi$-module over $\Lambda_{R,(0,b],K}$ is a coherent sheaf $D$ of  modules over the pseudorigid space $\bigcup_{a\rightarrow 0}\Spa(\Lambda_{R,[a,b],K})$ equipped with an isomorphism
	\[      \varphi_D\colon\varphi^\ast D\xrightarrow{\sim}\Lambda_{R,(0,b/p],K}\otimes_{\Lambda_{R,(0,b],K}}D   \]
	If $a\in (0,b/p]$, a $\varphi$-module over $\Lambda_{R,[a,b],K}$ is a finite $\Lambda_{R,[a,b],K}$-module $D$ equipped with an isomorphism
	\[      \varphi_{D,[a,b/p]}\colon\Lambda_{R,[a,b/p],K}\otimes_{\Lambda_{R,[a/p,b/p],K}}\varphi^\ast D\xrightarrow{\sim} \Lambda_{R,[a,b/p],K}\otimes_{\Lambda_{R,[a,b],K}}D  \]

	A $(\varphi,\Gamma_K)$-module over $\Lambda_{R,(0,b],K}$ (resp. $\Lambda_{R,[a,b],K}$) is a $\varphi$-module over $\Lambda_{R,(0,b],K}$ (resp. $\Lambda_{R,[a,b],K}$) equipped with a semi-linear action of $\Gamma_K$ which commutes with $\varphi_D$ (resp. $\varphi_{D,[a,b/p]}$).

	A $(\varphi,\Gamma_K)$-module over $R$ is a module $D$ over $\Lambda_{R,\rig,K}$ which arises via base change from a $(\varphi,\Gamma_K)$-module over $\Lambda_{R,(0,b],K}$ for some $b>0$.
\end{definition}
If $D$ is a $(\varphi,\Gamma)$-module over $\Lambda_{R,(0,b],K}$, and $I\subset (0,b]$ is a sub-interval, we will write $D_I\coloneqq  \Lambda_{R,I,K}\otimes_{\Lambda_{R,(0,b],K}}D$ to denote its restriction to the annulus $I$.

Let $K/L$ be a finite extension, and let $D$ be a $(\varphi,\Gamma_L)$-module over $\Lambda_{R,(0,b],K}$ (resp. $R$). As in ~\cite[\textsection 2.2]{liu2007}, we may define the induced $(\varphi,\Gamma_K)$-module $\Ind_L^K(D)$.  The underlying sheaf (resp. module) of $\Ind_L^K(D)$ is just $D$ itself, but viewed now as a sheaf on $\bigcup_{a\rightarrow 0}\Spa(\Lambda_{R,[a,b],K})$ (resp. a module over $\Lambda_{R,\rig,K}$).  If $D$ is projective, then ~\cite[Lemma 3.45]{bellovin2020} implies that $\Ind_L^K(D)$ is projective as well.

Let $\Delta_K\subset \Gamma_K$ be a torsion subgroup.  Since we assume $p\neq 2$, the quotient $\Gamma_K/\Delta_K$ is procyclic, so we may fix $\gamma\in\Gamma_K$ whose image in $\Gamma/\Delta_K$ is a topological generator.  Then for a $(\varphi,\Gamma_K)$-module $D$, we define the Fontaine--Herr--Liu complex via
\[      C_{\varphi,\Gamma}^\bullet\colon D\xrightarrow{(\varphi_D-1,\gamma-1)} D\oplus D\xrightarrow{(\gamma-1)\oplus(1-\varphi_D)} D      \]
(concentrated in degrees $0$, $1$, and $2$).  We let $H_{\varphi,\Gamma_K}^i(D)$ denote its cohomology in degree $i$.  

Then as in ~\cite[\textsection 2.1]{liu2007}, the projection $p_\Delta\colon D\rightarrow D^{\Delta_K}$ induces a quasi-isomorphism $C_{\varphi,\Gamma}^\bullet(D)\xrightarrow\sim C_{\varphi,\Gamma}^\bullet(D^{\Delta_K})$ (where we view $D^{\Delta_K}$ as a $(\varphi,\Gamma_K/\Delta_K)$-module over $\left(\Lambda_{R,(0,b],K}\right)^{\Delta_K}$).  In particular, it is independent of the choice of $\Delta_K$.

The main result of ~\cite{bellovin2020} says that if $M$ is a $R$-linear representation of $\Gal_K$, there is an associated projective $(\varphi,\Gamma_K)$-module $D_{\rig,K}(M)$. Moreover, we have a canonical quasi-isomorphism $R\Gamma(\Gal_K,M)\xrightarrow\sim C_{\varphi,\Gamma}^\bullet$ between (continuous) Galois cohomology and Fontaine--Herr--Liu cohomology.  This extends similar results on families of projective Galois representations with coefficients in classical $\Q_p$-affinoid algebras~\cite[Theorem 2.8]{pottharst2013} and earlier work in the setting of $\Q_p$-linear Galois representations~\cite[Theorem 2.3]{liu2007}.

We will define a closely related complex
\[	C_{\psi,\Gamma}^\bullet\colon D\xrightarrow{(\psi_D-1,\gamma-1)} D\oplus D\xrightarrow{(\gamma-1)\oplus(1-\psi_D)} D      \]
which is also concentrated in degrees $0$, $1$, and $2$.  

We first extend the $\psi$ operator to $(\varphi,\Gamma)$-modules.  The isomorphism $\varphi^\ast D_{(0,b]}\xrightarrow\sim D_{(0,b/p]}$ induces an isomorphism
\[      \Lambda_{R,(0,b/p],K}\otimes_{\varphi(\Lambda_{R,(0,b],K})}\varphi(D_{(0,b]})\xrightarrow\sim D_{(0,b/p]}       \]
We therefore have a surjective homomorphism $\psi_D\colon D_{(0,b/p]}\rightarrow D_{(0,b]}$ defined by setting $\psi_D(a\otimes\varphi(d))=\psi_D(a)d$, where $a\in \Lambda_{R,(0,b/p],K}$ and $d\in D_{(0,b]}$. 

There is a morphism of complexes $\Psi_D\colon C_{\varphi,\Gamma}^\bullet\rightarrow C_{\psi,\Gamma}^\bullet$ given by
\[
	\begin{tikzcd}
		C_{\varphi,\Gamma}^\bullet\colon \arrow{d}{\Psi_D} & D \arrow{r}\arrow{d}{\id} & D\oplus D\arrow{r}\arrow{d}{-\psi\oplus\id} & D\arrow{d}{-\psi}	\\
		C_{\psi,\Gamma}^\bullet\colon & D\arrow{r} & D\oplus D\arrow{r} & D
	\end{tikzcd}
\]

The following result is standard (see e.g. ~\cite[Proposition 2.3.6]{kpx}), and the same proof holds here:
\begin{lemma}
	The morphism $\psi_D$ is a quasi-isomorphism.
	\label{lemma: phi gamma vs psi gamma coh}
\end{lemma}

When there is no danger of confusion, we will generally drop the subscripts on $\varphi_D$ and $\psi_D$.

\subsection{\texorpdfstring{$(\varphi,\Gamma)$}{(𝜑, Γ)}-modules of character type}

Let $K/\Q_p$ be a finite extension with ramification degree $e_K$ and inertia degree $f_K$, and let $\mathscr{O}_K$ be its ring of integers, $k_K$ be its residue field, and $\varpi_K$ be a uniformizer.  Let $K_0\subset K$ be its maximal unramified subfield.  Let $R$ be a pseudoaffinoid algebra over $\Z_p$ with ring of definition $R_0\subset R$ and pseudouniformizer $u\in R_0$.  

We begin by recalling the construction of $(\varphi,\Gamma_K)$-modules of character type from~\cite{kpx}.
\begin{lemma}
	Let $\alpha\in R^\times$.  Up to isomorphism, there is a unique rank-$1$ $R\otimes_{\Z_p} \mathscr{O}_{K_0}$-module $D_{f_K,\alpha}$ equipped with a $1\otimes\varphi$-semilinear operator $\varphi_\alpha$ such that $\varphi_\alpha^{f_K}=\alpha\otimes 1$.
\end{lemma}
\begin{proof}
This follows exactly as in~\cite[Lemma 6.2.3]{kpx}.
\end{proof}

\begin{definition}
Let $\delta\colon K^\times \rightarrow R^\times$, and write $\delta=\delta_1\delta_2$, where $\delta_1,\delta_2\colon K^\times\rightrightarrows R^\times$ are continuous characters such that $\delta_1$ is trivial on $\mathscr{O}_K^\times$ and $\delta_2$ is trivial on $\left\langle\varpi_K\right\rangle$.  By local class field theory, $\delta_2$ corresponds to a continuous character $\delta_2'\colon\Gal_K\rightarrow R^\times$.  We let $\Lambda_{R,\rig,K}(\delta_1)\coloneqq D_{f_K,\delta_1(\varpi_K)}\otimes_{R\otimes \mathscr{O}_{K_0}}\Lambda_{R,\rig,K}$ and $\Lambda_{R,\rig,K}(\delta_2)\coloneqq D_{\rig,K}(\delta_2')$, and we define $\Lambda_{R,\rig,K}(\delta)\coloneqq \Lambda_{R,\rig,K}(\delta_1)\otimes\Lambda_{R,\rig,K}(\delta_2)$.
\end{definition}
If $D$ is a $(\varphi,\Gamma_K)$-module and $\delta\colon K^\times\rightarrow R^\times$ is a continuous character, we will let $D(\delta)$ denote $D\otimes \Lambda_{R,\rig,K}(\delta)$.  This is, in particular, a projective $(\varphi,\Gamma_K$-module of rank $1$.  We will let $C_{\varphi,\Gamma_K}^\bullet(\delta)$ and $H_{\varphi,\Gamma_K}^i(\delta)$ denote the Fontaine--Herr--Liu complex and the cohomology groups of $\Lambda_{R,\rig,K}(\delta)$, respectively.

\begin{lemma}\label{lemma: induction of character}
	Suppose $L/K$ is a finite extension, and $\varpi_L$ is a uniformizer of $L$ with $\Nm_{L/K}(\varpi_L)=\varpi_K$.  If $\delta\colon K^\times\rightarrow R^\times$ is a continuous character, then $\Res_K^L\Lambda_{R,\rig,K}(\delta)$ is of character type, with associated character $\delta\circ\Nm_L/K$.
\end{lemma}
\begin{proof}
	We may consider separately the cases where $\delta$ is trivial on $\mathscr{O}_K^\times$ and $\left\langle \varpi_K\right\rangle$.  If $\delta$ is trivial on $\mathscr{O}_K^\times$, then 
	\begin{equation*}
		\begin{split}
	\Res_K^L\Lambda_{R,\rig,K}(\delta) &= D_{f_K,\delta(\varpi_K)}\otimes_{R\otimes\mathscr{O}_{K_0}}\Lambda_{R,\rig,L} \\
	&= \left(D_{f_K,\delta(\varpi_K)}\otimes_{R\otimes\mathscr{O}_{K_0}}(R\otimes\mathscr{O}_{L_0})\right)\otimes_{R\otimes\mathscr{O}_{L_0}}\Lambda_{R,\rig,L}
		\end{split}
	\end{equation*}
	But $D_{f_K,\delta(\varpi_K)}\otimes_{R\otimes\mathscr{O}_K}(R\otimes\mathscr{O}_L)$ is a rank-$1$ $R\otimes\mathscr{O}_{L_0}$-module equipped with a $1\otimes\varphi$-semilinear operator $\varphi_{\delta(\varpi_K)}$ such that $\varphi_{\delta(\varpi_K)}^{f_L}=\delta(\varpi_K)\otimes 1$, so it is isomorphic to $D_{f_L,\delta(\varpi_K)}$. By definition, $D_{f_L,\delta(\varpi_K)}\otimes_{R\otimes\mathscr{O}_{L_0}}\Lambda_{R,\rig,L}$ is equal to $\Lambda_{R,\rig,L}(\delta\circ\Nm_{L/K})$.  On the other hand, if $\delta$ is trivial on $\left\langle \varpi_K\right\rangle$, the statement follows from functoriality for local class field theory.
\end{proof}

Continuous characters vary in analytic families, and hence $(\varphi,\Gamma)$-modules do, as well.  More precisely, if $G$ is a compact commutative $p$-adic Lie group, then the functor on complete sheafy affinoid $(\Z_p,\Z_p)$-algebras
\[	(A,A^+)\mapsto \Hom_{\cts}(G,A^\times)	\]
is representatable by the affinoid ring $(\Z_p[\![G]\!],\Z_p[\![G]\!])$.  Indeed, $G$ is non-canonically isomorphic to $G_0\times\Z_p^{\oplus r}$; if $\kappa\colon G\rightarrow A^\times$ is a continuous character, then $\kappa$ must carry $G_0$ to the roots of unity of $A^\times$ and must carry topological generators of $\Z_p^{\oplus r}$ to $1+A^{\circ\circ}$ (where $A^{\circ\circ}\subset A$ denotes the topologically nilpotent elements).

On the other hand, if $G$ is a free abelian group, then the affinoid adic space $\Spa(\Z[G],\Z)$ represents the functor on adic spaces
\[	X\mapsto \Hom(G,\mathscr{O}(X))	\]
Since $G$ is non-canonically isomorphic to $\Z^{\oplus r}$, this space is isomorphic to $\Gm^{\ad,r}$, where $\Gm^{\ad}\coloneqq \Spa(\Z[T^{\pm 1}],\Z)$ (note that $\Spa(\Z[T^{\pm 1}],\Z[T^{\pm 1}])$ is the ``unit circle'' representing the functor $X\mapsto \mathscr{O}^+(X)^\times$).
\begin{prop}
	If $X$ is an pseudorigid space, then the fiber product $X\times_{(\Z,\Z)} \Gm^r$ is representable by an analytic adic space $\G_{m,X}^r$.  If $X$ is pseudorigid, then so is $\G_{m,X}^r$.
\end{prop}
\begin{proof}
	We may assume that $r=1$.  To see that $\G_{m,X}$ is representable, we first assume that $X=\Spa R$ is affinoid, where $R$ is pseudoaffinoid with ring of definition $R_0$ and pseudouniformizer $u\in R_0$.  Then for any $h\in \Z_{\geq 0}$, we let $C_{X,h}\coloneqq \Spa R\left\langle u^hT, u^hT^{-1}\right\rangle$ denote the relative annulus; $C_{X,h}$ is again Tate, there are natural open immersions $C_{X,h}\subset C_{X,h+1}$, and we claim that $\cup_h C_{X,h}$ represents $\G_{m,X}$ in the category of adic spaces over $X$.  Indeed, if $\Spa A$ is an affinoid adic space over $X$, it is also Tate, and this claim amounts to the assertion that for a unit $f\in{A}^\times$, the set $\{f,f^{-1}\}\subset A$ is bounded.  But this follows from the definition of a pseudouniformizer.  Moreover, the union $\cup_h C_{X,h}$ is independent of the choice of $u$, even though the individual annuli do depend on $u$.  Then gluing shows that $\G_{m,X}$ is representable for general $X$.
\end{proof}

\begin{definition}
	Suppose that $G$ is a commutative $p$-adic Lie group, and $G'\subset G$ is a compact open subgroup such that $G/G'$ is free and finitely generated.  Then we define the pseudorigid character variety $\widehat G^{\an}\coloneqq \Spa(\Z_p[\![G']\!],\Z_p[\![G']\!])^{\an}\times_{(\Z,\Z)}\Spa(\Z[G/G'],\Z)$; if $X$ is an arbitrary pseudorigid space, we also define $\widehat G_X\coloneqq \Spa(\Z_p[\![G']\!],\Z_p[\![G']\!])\times_{(\Z_p,\Z_p)}\Spa(\Z[G/G'],\Z)_X$ (which is also pseudorigid).
	\label{def:G hat}
\end{definition}
In order to make sense of the last definition, we observe that if $R$ is a pseudoaffinoid algebra with noetherian ring of definition $R_0\subset R$ and pseudouniformizer $u\in R_0$, then $\widehat{G'}\times_{(\Z_p,\Z_p)}\Spa R$ is the (non-quasi-compact) open subspace $\{u\neq 0\}\subset \Spa R_0\htimes \Z_p[\![G']\!]$.

In particular, if $K$ is a finite extension of $\Q_p$, we have the pseudorigid moduli space $\widehat{K^\times}$ of continuous characters of $K^\times$.  If $K=\Q_p$ and $G=\Q_p^\times\cong \mu_{p-1}\times\Z\times\Z_p$, then $\widehat G^{\an}$ has connected components indexed by the elements of $\mu_{p-1}$, each of which is isomorphic to $\left(\Spa \Z_p[\![\Z_p]\!]\right)^{\an}\times\G_m^{\ad}$.

\begin{remark}
	In the pseudorigid setting (unlike the classical rigid analytic setting), it is not true that $\widehat{G_1\times G_2}^{\an}\cong \widehat{G_1}^{\an}\times \widehat{G_2}^{\an}$.  Indeed, $\Spa(\Z_p[\![T_1,T_2]\!],\Z_p[\![T_1,T_2]\!])^{\an}$ consists of all valuations which do not vanish on all three of $p, T_1, T_2$.  But 
	\[	\Spa(\Z_p[\![T_1]\!],\Z_p[\![T_1]\!])^{\an}\times_{(\Z_p,\Z_p)}\Spa(\Z_p[\![T_2]\!],\Z_p[\![T_2]\!])^{\an}	\]
	also excludes valuations vanishing at both $p$ and $T_1$ (or both $p$ and $T_2$).
	\label{rmk: products char varieties}
\end{remark}

\section{Finiteness of cohomology}

We wish to show that the cohomology of $C_{\varphi,\Gamma}^\bullet$ is $R$-finite.  To do this, we will apply~\cite[Lemma 1.10]{kedlaya-liu16} to the morphisms of complexes
\[
\xymatrix{
D_{[a,b]} \ar[r]\ar[d] & D_{[a,b/p]}\oplus D_{[a,b]}\ar[r]\ar[d] & D_{[a,b/p]}\ar[d]        \\
D_{[a',b']}\ar[r] & D_{[a',b'/p]}\oplus D_{[a',b']}\ar[r] & D_{[a',b'/p]}
}\]
induced by the natural homomorphisms ${\Lambda}_{R,[a,b],K'}\rightarrow{\Lambda}_{R,[a',b'],K'}$ (where $[a',b']\subset (a,b)$).  More precisely, Kedlaya--Liu show that if the morphisms $D_{[a,b]}\rightarrow D_{[a',b']}$ are completely continuous and induce isomorphisms on cohomology groups, then both complexes have $R$-finite cohomology.  Since $C_{\varphi,\Gamma}^\bullet$ is the direct limit (as $b\rightarrow 0$) of the inverse limit of these complexes (as $a\rightarrow 0$), with transition maps which are quasi-isomorphisms, this will imply that $C_{\varphi,\Gamma}^\bullet$ has $R$-finite cohomology.

\begin{definition}
	Let $A$ be a ring.  A function $\lvert\cdot\rvert\colon A\rightarrow \R_{\geq 0}$ is called a \emph{semi-norm} if
	\begin{enumerate}
		\item	$\lvert 0\rvert = 0$ and $\lvert 1\rvert = 1$
		\item	$\lvert a + b\rvert \leq \max\{\lvert a\rvert, \lvert b\rvert\}$ for all $a, b\in R$
		\item	$\lvert ab\rvert\leq \lvert a\rvert\cdot\lvert b\rvert$ for all $a, b\in R$
	\end{enumerate}
	If, in addition, $\lvert a\rvert = 0$ if and only if $a=0$, we say that $\lvert \cdot \rvert$ is a \emph{norm}.

	We say that $A$ is a \emph{Banach algebra} if $A$ is equipped with a norm, and it is complete with respect to the metric induced by that norm.
	\label{def: banach algebra}
\end{definition}

If $A$ is a complete Tate ring with ring of definition $A_0$ and pseudouniformizer $u\in A_0$, we may define a norm on $A$ as follows:  Let $\alpha\in \R_{>1}$, and define
\[	\lvert a\rvert = \inf\{\alpha^{-n} \mid a\in u^{n}A_0\}	\]

We recall \cite[Definition 1.3]{kedlaya-liu16}:
\begin{definition}\label{def: completely cont}
Let $A$ be a Banach algebra, and let $f\colon M\rightarrow N$ be a morphism of Banach $A$-modules (equipped with norms $\lvert\cdot\rvert_M$ and $\lvert\cdot\rvert_N$, respectively).  We say that $f$ is \emph{completely continuous} if there exists a sequence of finite $A$-submodules $N_i$ of $N$ such that the operator norms of the compositions $M\rightarrow N\rightarrow N/N_i$ tend to $0$ (where $N/N_i$ is equipped with the quotient semi-norm)
\end{definition}
Note that this is slightly different than the standard definition (cf. the discussion of ~\cite[Remark 1.12]{kedlaya-liu16}).

We also recall ~\cite[Definition 5.1]{kedlaya-liu16}:
\begin{definition}
Let $f\colon(A,A^+)\rightarrow (A',{A'}^+)$ be a localization of complete Tate rings over a complete Tate ring $(B,B^+)$.  We say that $f$ is \emph{inner} if there is a strict $B$-linear surjection $B\left\langle \underline X\right\rangle\twoheadrightarrow A$ such that each element of $\underline X$ maps to a topologically nilpotent element of $A'$.  Here $\underline X$ is a (possibly infinite) collection of formal variables.
\end{definition}

If $B$ is a nonarchimedean field of mixed characteristic and $A$ and $A'$ are topologically of finite type over $B$, Kiehl proved that inner homomorphisms are completely continuous.  We prove the analogous result, using Definition~\ref{def: completely cont} as the definition of complete continuity (which is slightly different than Kiehl's definition).

\begin{prop}
If $[a',b']\subset (a,b)$ and $[a,b]\subset (0,\infty)$, then the map $\Lambda_{R,[a,b],K}\rightarrow \Lambda_{R,[a',b'],K}$ induced by restriction is completely continuous.
\end{prop}
\begin{proof}
The pairs $(\Lambda_{R,[a,b],K},\Lambda_{R,[a,b],K}^+)$ and $(\Lambda_{R,[a',b'],K},\Lambda_{R,[a',b'],K}^+)$ are localizations of $(R_0\otimes\mathscr{O}_{F'}[\![\pi_K]\!],R_0\otimes\mathscr{O}_{F'}[\![\pi_K]\!])$; since $[a',b'],[a,b]\subset(0,\infty)$, they are adic affinoid algebras over $(R,R^+)$.  Since $[a',b']\subset (a,b)$, the natural restriction map is inner.  Then~\cite[Lemma 5.7]{kedlaya-liu16} implies that it is completely continuous.
\end{proof}

\begin{lemma}
Suppose $a\in (0,b/p]$.  Then the functor $D\rightsquigarrow \Lambda_{R,[a,b],K}\otimes_{\Lambda_{R,(0,b],K}}D=:D_{[a,b]}$ induces an equivalence of categories between $\varphi$-modules over $\Lambda_{R,(0,b],K}$ and $\varphi$-modules over $\Lambda_{R,[a,b],K}$.
\end{lemma}
\begin{proof}
Suppose we have a $\varphi$-module $D_{[a,b]}$ over $\Lambda_{R,[a,b],K}$.  Then the Frobenius pullback $\varphi^\ast D_{[a,b]}$ is a finite module over $\Lambda_{R,[a/p,b/p],K}$, and the isomorphism $\varphi_{D,[a,b/p]}\colon \Lambda_{R,[a,b/p],K}\otimes_{\Lambda_{R,[a/p,b/p],K}}\varphi^\ast D_{[a,b]}\xrightarrow{\sim} \Lambda_{R,[a,b/p],K}\otimes_{\Lambda_{R,[a,b],K}}D_{[a,b]}$ (and the assumption that $a\leq b/p$) provides a descent datum.  Thus, we may construct a finite module $D_{[a/p,b]}$ over $\Lambda_{R,[a/p,b],K}$ which restricts to $D_{[a,b]}$.

To show that $D_{[a/p,b]}$ is a $\varphi$-module over $\Lambda_{R,[a/p,b],K}$, we need to construct an isomorphism
\[      \varphi_{D,[a/p,b/p]}\colon\Lambda_{R,[a/p,b/p],K}\otimes_{\Lambda_{R,[a/p^2,b/p],K}}\varphi^\ast D_{[a/p,b]}\xrightarrow\sim \Lambda_{R,[a/p,b/p],K}\otimes_{\Lambda_{R,[a/p,b],K}} D_{[a/p,b]}     \]
By construction, we have an isomorphism
\[      \varphi_{D,[a,b/p]}\colon\varphi^\ast D_{[a,b]}\xrightarrow{\sim} \Lambda_{R,[a/p,b/p],K}\otimes_{\Lambda_{R,[a/p,b],K}}D_{[a/p,b]}  \]
and if we pull $\varphi_{D,[a,b/p]}$ back by Frobenius, we obtain an isomorphism
\begin{equation*}
\resizebox{\displaywidth}{!}{
$\varphi_{D,[a/p,b/p^2]}\colon \Lambda_{R,[a/p,b/p^2],K}\otimes_{\Lambda_{R,[a/p^2,b/p^2],K}}\varphi^\ast D_{[a/p,b/p]}\xrightarrow{\sim} \Lambda_{R,[a/p,b/p^2],K}\otimes_{\Lambda_{R,[a/p,b/p],K}}D_{[a/p,b/p]}    $}
\end{equation*}
On the overlap, they induce the same isomorphism $\varphi^\ast D_{[a,b/p]}\rightarrow \Lambda_{R,[a/p,b/p^2],K}\otimes_{\Lambda_{R,[a/p,b/p],K}}D_{[a/p,b/p]}$ (by construction), so we obtain the desired isomorphism $\varphi_{D,[a/p,b/p]}$.

Iterating this construction lets us construct a $\varphi$-module over $\Lambda_{R,(0,b],K}$.

This proves essential surjectivity; full faithfulness follows because the natural maps $\Lambda_{R,(0,b],K}\rightarrow\Lambda_{R,[a,b],K}$ have dense image.
\end{proof}

\begin{cor}
If $D$ is a $\varphi$-module over $\Lambda_{R,(0,b],K}$, the morphism of complexes
\[      [D_{(0,b]}\xrightarrow{\varphi-1}D_{(0,b/p]}]\rightarrow [D_{[a,b]}\xrightarrow{\varphi-1}D_{[a,b/p]}]    \]
is a quasi-isomorphism for any $a\in (0,b/p]$.
\end{cor}
\begin{proof}
This follows from the previous lemma because we may interpret the cohomology groups as Yoneda Ext groups.
\end{proof}

In order to prove that the restriction map $D_{(0,b]}\rightarrow D_{(0,b/p]}$ induces an isomorphism on cohomology, we will need to use the $\psi$ operator:
\begin{lemma}\label{lemma: gamma - 1 inverse}
Let $D$ be a $(\varphi,\Gamma)$-module over $\Lambda_{R,(0,b],K}$ for some $b>0$.  Then there is some $0<b'\leq b$ such that the action of $\gamma-1$ on $(D_{(0,b']})^{\psi=0}$ admits a continuous inverse.
\end{lemma}
\begin{proof}
We may replace $D$ with $\Ind_K^{\Q_p}(D)$.  Since $(D_{(0,b/p^n]})^{\psi=0}=\oplus_{j\in(\Z/p^n)^\times}[\varepsilon]^{\tilde j}\varphi^n(D)$, it suffices to show that $\gamma-1$ has a continuous inverse on $[\varepsilon]^{j}\varphi^n(D)$ for $j$ prime to $p$ and sufficiently large $n$.  Moreover, since $\gamma^n-1=(\gamma-1)(\gamma^{n-1}+\cdots+1)$, we may replace $\Gamma_{\Q_p}$ with a finite-index subgroup.  

If $\gamma_n\in\Gamma_{\Q_p}$ is such that $\chi(\gamma_n)=1+p^n$, then
\begin{equation*}
\begin{split}
\gamma_n\left([\varepsilon]^j\varphi^n(x)\right)-[\varepsilon]^j\varphi^n(x) &= [\varepsilon]^j[\varepsilon]^{p^nj}\varphi^n(\gamma_n(x)) - [\varepsilon]^j\varphi^n(x) \\
&=[\varepsilon]^j\varphi^n([\varepsilon]^j\gamma_n(x)-x)        \\
&=[\varepsilon]^j\varphi^n(G_{\gamma_n}(x))
\end{split}
\end{equation*}
where $G_{\gamma_n}(x)\coloneqq [\varepsilon]^j\gamma_n(x)-x = ([\varepsilon]^j-1)\cdot\left(1+\frac{[\varepsilon]^j}{[\varepsilon]^j-1}(\gamma_n-1)\right)(x)$.  Thus, if we can choose $n$ such that $\sum_{k=0}^\infty\left(-\frac{[\varepsilon]^j}{[\varepsilon]^j-1}(\gamma_n-1)\right)^k$ converges on $D_{(0,b]}$, we will be done.

The action of $\Gamma_{\Q_p}$ on $D_{[b/p,b]}$ is continuous, so we may choose $n_0$ such that for $n\geq n_0$, the sum above converges in $\mathrm{End}(D_{[b/p,b]})$.  But $D_{[b/p^{k+1},b/p^k]}\cong \varphi^\ast D_{[b/p^k,b/p^{k-1}]}$ for all $k\geq 0$ and the action of $\Gamma$ commutes with the action of $\varphi$, so the sum converges in $\mathrm{End}(D_{[b/p^{k+1},b/p^k]})$ for all $k\geq 0$ and $\gamma-1$ acts invertibly on $D_{(0,b/p^n]}^{\psi=0}$ for $n\geq n_0$.
\end{proof}

\begin{prop}\label{prop: coh move b}
If $D$ is a $(\varphi,\Gamma)$-module over $R$, then the cohomology of $D$ is computed by
\[      D_{[a,b]}\xrightarrow{(\varphi-1,\gamma-1)} D_{[a,b/p]}\oplus D_{[a,b]}\xrightarrow{(\gamma-1)\oplus(1-\varphi)} D_{[a,b/p]}  \]
for some sufficiently small $b$ and any $a\in (0,b/p]$. 
\end{prop}
\begin{proof}
We may assume that $D$ is a $(\varphi,\Gamma)$-module over $\Lambda_{R,(0,b],K}$ for some $b>0$.  Since $[D_{(0,b]}\xrightarrow{\varphi-1}D_{(0,b/p]}]\rightarrow [D_{[a,b]}\xrightarrow{\varphi-1}D_{[a,b/p]}]$ induces an isomorphism on cohomology, we see that the cohomology of $D_{(0,b]}$ is computed by the above complex.

Since the cohomology of $C_{\varphi,\Gamma}^\bullet(D)$ is computed by the direct limit of the cohomology groups of $C_{\varphi,\Gamma}^\bullet(D_{(0,b/p^n]})$ as $n\rightarrow \infty$, it suffices to show that the natural morphism
\[
\xymatrix{
C_{(0,b]}^\bullet\colon &{D}_{(0,b]} \ar[r]\ar[d] & D_{(0,b/p]}\oplus D_{(0,b]}\ar[r]\ar[d] & D_{(0,b/p]}\ar[d]  \\
C_{(0,b/p]}^\bullet\colon &D_{(0,b/p]}\ar[r] & D_{(0,b/p^2]}\oplus D_{(0,b/p]}\ar[r] & D_{(0,b/p^2]}
}\]
induces an isomorphism on cohomology groups for sufficiently small $b$.

We first show that the morphisms $1,\varphi\colon C_{(0,b]}^\bullet\rightrightarrows C_{(0,b/p]}^\bullet$ are homotopic.  This follows by considering the diagram
\[
\xymatrixcolsep{5pc}\xymatrix{
	{D}_{(0,b]} \ar[r]^-{(\varphi-1,\gamma-1)}\ar@<-.5ex>[d]_{1}\ar@<.5ex>[d]^{\varphi} & D_{(0,b/p]}\oplus D_{(0,b]}\ar[r]^-{(\gamma-1)\oplus(1-\varphi)}\ar@<-.5ex>[d]_{1}\ar@<.5ex>[d]^{\varphi}\ar[ld]_{\pr_1} & D_{(0,b/p]}\ar@<-.5ex>[d]_{1}\ar@<.5ex>[d]^{\varphi}\ar[ld]_{(0,-1)}       \\
D_{(0,b/p]}\ar[r]^-{(\varphi-1,\gamma-1)} & D_{(0,b/p^2]}\oplus D_{(0,b/p]}\ar[r]^-{(\gamma-1)\oplus(1-\varphi)} & D_{(0,b/p^2]}
}\]
Thus, it suffices to show that the morphism $\varphi\colon C_{(0,b]}^\bullet\rightarrow C_{(0,b/p]}^\bullet$ is a quasi-isomorphism.  But the cokernel (in the category of complexes) is the complex
\[	D_{(0,b/p]}^{\psi=0}\xrightarrow{(-1,\gamma-1)}D_{(0,b/p]}^{\psi=0}\oplus D_{(0,b/p]}^{\psi=0} \xrightarrow{(\gamma-1)\oplus 1}D_{(0,b/p]}^{\psi=0}	\]
Since $\gamma-1$ acts invertibly on $D_{(0,b/p]}^{\psi=0}$, the cohomology of this complex vanishes and the result follows.
\end{proof}

\begin{cor}
If $D$ is a $(\varphi,\Gamma)$-module over $\Lambda_{R,(0,b],K}$ and $[a',b']\subset [a,b]$ and $b$ is sufficiently small, the restriction map
\[
\xymatrix{
D_{[a,b]}\ar[r]\ar[d] & D_{[a,b/p]}\oplus D_{[a,b]}\ar[r]\ar[d] & D_{[a,b/p]}\ar[d]     \\
D_{[a',b']}\ar[r] & D_{[a',b'/p]}\oplus D_{[a',b']}\ar[r] & D_{[a',b'/p]}
}\]
induces an isomorphism on cohomology.
\end{cor}
\begin{proof}
We may assume that $b'\in [b/p,b]$, so that we have induced homomorphisms
\[      H_{\varphi,\Gamma}^i(D_{(0,b]})\rightarrow H_{\varphi,\Gamma}^i(D_{(0,b']})\rightarrow H_{\varphi,\Gamma}^i(D_{(0,b/p]})\rightarrow H_{\varphi,\Gamma}^i(D_{(0,b'/p]})  \]
Since the compositions $H_{\varphi,\Gamma}^i(D_{(0,b]})\rightarrow H_{\varphi,\Gamma}^i(D_{(0,b/p]})$ and $H_{\varphi,\Gamma}^i(D_{(0,b']})\rightarrow H_{\varphi,\Gamma}^i(D_{(0,b'/p]})$ are isomorphisms, the homomorphism $H_{\varphi,\Gamma}^i(D_{(0,b']})\rightarrow H_{\varphi,\Gamma}^i(D_{(0,b/p]})$ is also an isomorphism, and we are done.
\end{proof}

Now we can finally prove that $(\varphi,\Gamma)$-modules have finite cohomology.
\begin{thm}
If $D$ is a $(\varphi,\Gamma)$-module over $\Lambda_{R,(0,b],K}$, its cohomology is $R$-finite when $b$ is sufficiently small.
\end{thm}
\begin{proof}
If $[a',b']\subset (a,b)$, the restriction map induces a quasi-isomorphism
\[
\xymatrix{
D_{[a,b]}\ar[r]\ar[d] & D_{[a,b/p]}\oplus D_{[a,b]}\ar[r]\ar[d] & D_{[a,b/p]}\ar[d]     \\
D_{[a',b']}\ar[r] & D_{[a',b'/p]}\oplus D_{[a',b']}\ar[r] & D_{[a',b'/p]}
}\]
which is completely continuous. Then the result follows, by~\cite[Lemma 1.10]{kedlaya-liu16}.
\end{proof}

\begin{cor}\label{cor: coh perfect}
If $D$ is a projective $(\varphi,\Gamma)$-module over $R$, then $C_{\varphi,\Gamma_K}^\bullet(D)\in \mathbf{D}_{\mathrm{perf}}^{[0,2]}(R)$.
\end{cor}
\begin{proof}
	Finiteness of the cohomology of $C_{\varphi,\Gamma_K}^\bullet(D)$ implies that $C_{\varphi,\Gamma_K}^\bullet(D)\in \mathbf{D}_{\mathrm{perf}}^-(R)$, and by ~\cite[Proposition 3.47]{bellovin2020}, the complex $C_{\varphi,\Gamma_K}^\bullet(D)$ consists of flat $A$-modules.  Then as in the proof of ~\cite[Theorem 4.4.5(1)]{kpx}, it follows that $C_{\varphi,\Gamma_K}^\bullet(D)\in \mathbf{D}_{\mathrm{perf}}^{[0,2]}(R)$.
\end{proof}

\begin{cor}\label{cor: coh coh sheaf}
If $D$ is a projective $(\varphi,\Gamma)$-module over $R$, then the cohomology groups $H_{\varphi,\Gamma}^i(D)$ are coherent sheaves on $\Spa R$.
\end{cor}
\begin{proof}
	Since $C_{\varphi,\Gamma}^\bullet(D)\in D_{\mathrm{coh}}^b(R)$, we have a quasi-isomorphism $R'\otimes_R^{\mathbf{L}} C_{\varphi,\Gamma}^\bullet(D)\xrightarrow{\sim}C_{\varphi,\Gamma}^\bullet(R'\otimes_RD)$ for any homomorphism $R\rightarrow R'$ of pseudoaffinoid algebras.  If $R\rightarrow R'$ defines an affinoid subspace of $\Spa R$, the morphism is flat and the derived tensor product is an ordinary tensor product. On the other hand, we have a natural homomorphism $C_{\varphi,\Gamma}^\bullet(R'\otimes_RD)\rightarrow C_{\varphi,\Gamma}^\bullet(R'\htimes_RD)$, and it is a quasi-isomorphism after every specialization $R'\twoheadrightarrow S$ to a finite-length algebra (since $D$ is flat over $R$).  Then the result follows from ~\cite[Lemma 4.1.5]{kpx}.
\end{proof}

As a corollary, if $R\rightarrow R'$ is a homomorphism of pseudoaffinoid algebras, there is a natural quasi-isomorphism 
\[	R'\otimes^{\mathrm{L}}C_{\varphi,\Gamma}^\bullet(D)\xrightarrow{\sim}C_{\varphi,\Gamma}^\bullet(R'\otimes_RD)	\]
and there is a corresponding second-quadrant base-change spectral sequence.  We record the low-degree exact sequences of the base-change spectral sequence here:
\begin{cor}\label{cor: coh and base change}
Let $R\rightarrow R'$ be a morphism of pseudoaffinoid algebras and let $D$ be a $(\varphi,\Gamma_K)$-module over $\Lambda_{R,\rig,K}$.  Then
\begin{enumerate}
        \item   The natural morphism $R'\otimes_RH_{\varphi,\Gamma_K}^2(D)\rightarrow H_{\varphi,\Gamma_K}^2(R'\otimes_RD)$ is an isomorphism.
        \item   The natural morphism $R'\otimes_RH_{\varphi,\Gamma_K}^1(D)\rightarrow H_{\varphi,\Gamma_K}^1(R'\otimes_RD)$ fits into an exact sequence
		\[
			0\rightarrow \Tor_2^R(H_{\varphi,\Gamma_K}^2(D),R')\rightarrow R'\otimes_RH_{\varphi,\Gamma_K}^1(D)	
			\rightarrow H_{\varphi,\Gamma_K}^1(R'\otimes_RD)\rightarrow \Tor_1^R(H_{\varphi,\Gamma_K}^2(D),R')\rightarrow 0
		\]
	\item	If $H_{\varphi,\Gamma_K}^1(D)$ and $H_{\varphi,\Gamma_K}^2(D)$ have Tor-dimension at most $1$, then we have an exact sequence
		\[	0\rightarrow R'\otimes_RH_{\varphi,\Gamma}^0(D)\rightarrow H_{\varphi,\Gamma}^0(R\otimes_{R'}D)\rightarrow \Tor_1^R(H_{\varphi,\Gamma_K}^1(D), R')\rightarrow 0	\]

\end{enumerate}
\end{cor}
\begin{proof}
This follows from the convergence of the base-change spectral sequence.
\end{proof}

Since we have shown that $C_{\varphi,\Gamma_K}^\bullet(D)\in \D_{\mathrm{perf}}^{[0,2]}(R)$ is a perfect complex, we may define its Euler characteristic.  If $P^\bullet\in \D_{\mathrm{perf}}^{[a,b]}(R)$ is a complex of finite projective $R$-modules, we define the \emph{Euler characteristic}
\[	\chi(P^\bullet)\coloneqq \sum_{i=a}^b(-1)^i\rk P^i	\]
This is invariant under quasi-isomorphism and additive under distinguished triangles.  If $D$ is a $(\varphi,\Gamma_K)$-module, we simply write $\chi(D)$ for $\chi(P^\bullet)$, where $P^\bullet$ is any complex in $\D_{\mathrm{perf}}^{[0,2]}(R)$ quasi-isomorphic to $C_{\varphi,\Gamma_K}^\bullet(D)$.  Then we have the following:
\begin{cor}[Euler characteristic formula]\label{thm: euler characteristic}
	If $D$ is a projective $(\varphi,\Gamma_K)$-module with coefficients in a pseudoaffinoid algebra $R$, then $\chi(D)=-(\rk D)[K:\Q_p]$.
\end{cor}

We will use the slope filtration theorem of ~\cite[Theorem 1.7.1]{kedlaya08}, which holds if the coefficients are a finite extension of either $\Q_p$ or $\F_p(\!(u)!)$.  We first note the following:
\begin{lemma}
	If $D$ is a $(\varphi,\Gamma_K)$-module with coefficients in a field $R$ which is pure of slope $0$, then $D$ arises from a Galois representation.
	\label{lemma: etale phi-gamma galois rep}
\end{lemma}
\begin{proof}
	By ~\cite[Theorem 1.7.1]{kedlaya08}, the hypothesis implies that $D$ is \'etale, and hence (by ~\cite[Proposition 1.5.5]{kedlaya08}) arises from a $(\varphi,\Gamma_K)$-module over $\Lambda_{R_0,[0,b],\Q_p}$ for some $b>0$, and hence from a $(\varphi,\Gamma_K)$-module over $\Lambda_{R_0,[0,b],\Q_p,0}$.\mar{finish}  
\end{proof}

\begin{proof}[Proof of Corollary~\ref{thm: euler characteristic}]
	Euler characteristics are locally constant, so it suffices to compute $\chi(D_x)$ for a single maximal point $x$ on each connected component of $\Spa R$.  Thus, we may assume that $R$ is a finite extension of either $\Q_p$ or $\F_p(\!(u)\!)$.

Since Euler characteristics are additive in exact sequences, we may assume that $D$ is pure of slope $s$; if necessary, replace $R$ by an \'etale extension so that the slope of $D$ is in the value group of $R$.  The moduli space $X\coloneqq \widehat{\left\langle\varpi_K\right\rangle}_R\cong\G_{m,R}^{\an}$ of continuous characters of $\left\langle \varpi_K\right\rangle$ has a universal character $\delta_{\mathrm{univ}}\colon\left\langle\varpi_K\right\rangle\rightarrow \mathscr{O}(X)^\times$, so we may consider the Fontaine--Herr--Liu complex $C_{\varphi,\Gamma}^\bullet$ of the $(\varphi,\Gamma_K)$-module $D(\delta_{\mathrm{univ}})$ over $X$.  

Since $C_{\varphi,\Gamma}^\bullet(D(\delta))\in D_{\mathrm{perf}}^{[0,2]}(R')$ for every affinoid subdomain $\Spa(R')\subset X$, its Euler characteristic is constant on connected components of $X$, and it suffices to verify the statement at one point on each component.  But each connected component contains a point $x$ such that the slope of $D(\delta)$ at $x$ is $0$; then $(D(\delta))(x)$ is \'etale and by Lemma~\ref{lemma: etale phi-gamma galois rep} we may appeal to the Euler characteristic formula for Galois cohomology.
\end{proof}

In section ~\ref{section: tate local duality}, we will prove Tate local duality for projective $(\varphi,\Gamma)$-modules when $R$ is a finite extension of $\F_p(\!(u)\!)$.  We deduce the corresponding result for families of $(\varphi,\Gamma)$-modules over general pseudoaffinoid algebras here, and the reader may check that there is no circular dependence.
\begin{thm}[Tate local duality]
        Let $R$ be a pseudoaffinoid algebra and let $D$ be a family of projective $(\varphi,\Gamma_K)$-modules over $R$.  Then  the natural morphism
        \[      C_{\varphi,\Gamma_K}^\bullet(D)\rightarrow \R\Hom_R(C_{\varphi,\Gamma_K}^\bullet(D^\vee(\chi_{\cyc})),R)[-2]    \]
        is a quasi-isomorphism.
	\label{thm: tate duality}
\end{thm}
\begin{proof}
	For every maximal point $x\in\Spa R$, we have a quasi-isomorphism $C_{\varphi,\Gamma_K}^\bullet(R/\mathfrak{m}_x\otimes_RD)\xrightarrow{\sim}\R\Hom_R(C_{\varphi,\Gamma_K}^\bullet(R/\mathfrak{m}_x\otimes_R D^\vee(\chi_{\cyc})),R/\mathfrak{m}_x)$, by ~\cite[Theorem 4.7]{liu2007} (when $R/\mathfrak{m}_x$ has characteristic $0$) and Theorem ~\ref{thm: tate local duality} (when $R/\mathfrak{m}_x$ has positive characteristic).  Then by ~\cite[Lemma 4.1.5]{kpx}, the result follows.
\end{proof}

We conclude this section by recording the following result for later use; it is a corollary of the method of the proof of finiteness of cohomology, rather than finiteness itself.
\begin{cor}
	If $D$ is a projective $(\varphi,\Gamma_K)$-module over $R$, $f$ is the inertial degree of $K$, and $\alpha\in R$, then for $b$ sufficiently small and any $a<b/p^{f+1}$, the complex
	\[	D_{[a,b]}\xrightarrow{\varphi^f-\alpha,\gamma-1} D_{[a,b/p^f]}\oplus D_{[a,b]}\xrightarrow{(\gamma-1)\oplus(\alpha-\varphi^f)} D_{[a,b/p^f]}	\]
	is in $\mathbf{D}_{\mathrm{perf}}^{[0,2]}(R)$, and its cohomology groups form coherent sheaves on $\Spa R$.
	\label{cor: finiteness phi^f-alpha}
\end{cor}
\begin{proof}
	Choose some $a'\in (a,b/p^{f+1})$.  As in the proof of Proposition~\ref{prop: coh move b}, we consider the complex 
	\[
\xymatrixcolsep{5pc}\xymatrix{
        {D}_{[a,b]} \ar[r]^-{(\varphi^f-\alpha,\gamma-1)}\ar@<-.5ex>[d]_{\alpha}\ar@<.5ex>[d]^{\varphi^f} & D_{[a,b/p^f]}\oplus D_{[a,b]}\ar[r]^-{(\gamma-1)\oplus(\alpha-\varphi^f)}\ar@<-.5ex>[d]_{\alpha}\ar@<.5ex>[d]^{\varphi^f}\ar[ld]_{\pr_1} & D_{[a,b/p^f]}\ar@<-.5ex>[d]_{\alpha}\ar@<.5ex>[d]^{\varphi^f}\ar[ld]_{(0,-1)}       \\
	D_{[a',b/p^f]}\ar[r]^-{(\varphi^f-\alpha,\gamma-1)} & D_{[a',b/p^{f+1}]}\oplus D_{[a',b/p^f]}\ar[r]^-{(\gamma-1)\oplus(\alpha-\varphi^f)} & D_{[a',b/p^{f+1}]}
}\]
This shows that ``restrict and multiply by $\alpha$'' is homotopic to $\varphi^f$, which is a quasi-isomorphism.  Applying ~\cite[Lemma 1.10]{kedlaya-liu16} again, we see that the cohomology groups of our complex are $R$-finite.  We conclude as in the proofs of Corollary ~\ref{cor: coh perfect} and Corollary~\ref{cor: coh coh sheaf}.
\end{proof}

\section{Positive characteristic function fields}

In this section, we closely study overconvergent $(\varphi,\Gamma)$-modules where the coefficients are finite extensions of $\F_p(\!(u)\!)$; throughout this section, $R$ will denote such an extension.  This is similar to the situation studied by Hartl--Pink~\cite{hartl-pink}, but because we are interested in $(\varphi,\Gamma)$-modules related to representations of characteristic-$0$ Galois groups, we may work with imperfect coefficients.  For this reason, we rely on the slope filtration theorem from~\cite{kedlaya08}, rather than the Dieudonn\'e--Manin classification theorem from ~\cite{hartl-pink}.  We first calculate the cohomology of certain rank-$1$ $(\varphi,\Gamma_{\Q_p})$-modules (using techniques similar to ~\cite{colmez2005}), and then use those calculations to deduce the Tate local duality theorem for all $(\varphi,\Gamma_K)$-modules (following the strategy of ~\cite{liu2007}).

\subsection{Cohomology of rank-$1$ \texorpdfstring{$(\varphi,\Gamma)$}{(𝜑, Γ)}-modules}

We begin by computing the cohomology of certain distinguished $(\varphi,\Gamma_K)$-modules of character type.  When $K=\Q_p$, we let $\Delta_{\Q_p}\coloneqq \mu_{p-1}$ be the maximal torsion subgroup of $\Gamma_{\Q_p}$.

\begin{lemma}\label{lemma: h0 character}
Let $\delta\colon K^\times\rightarrow R^\times$ be a continuous character.  Then $H_{\varphi,\Gamma_K}^0(\delta)=0$ unless $\delta$ is trivial, in which case it is a free $R$-module of rank $1$.
\end{lemma}
\begin{proof}
Write $\delta=\delta_1\delta_2$, where $\delta_1$ is trivial on $\O_K^\times$ and $\delta_2$ is trivial on $\left\langle\varpi_K\right\rangle$.  We first show that the kernel of $\varphi-1$ on $\Lambda_{R,\rig,K}(\delta_1)$ is trivial unless $\delta_1(\varpi_K)=1$, in which case it is $R$, and then compute the elements of $\ker(\varphi-1)$ fixed by $\Gamma_K$.

If $f(\overline\pi_K)\in \Lambda_{R,\rig,K}(\delta_1)$, we may write $f(\overline\pi_K)$ uniquely as $f(\overline\pi_K)=\sum_{i\in\Z}a_i\overline\pi_K^i$, where $a_i\in R\otimes k'$ (for some finite extension $k'/k_K$).  There is some integer $f\geq 1$ such that that $\varphi^{f_Kf}$ fixes $k'$.  Using the fact that $\varphi^{f_Kf}(\overline\pi_K)=\overline\pi_K^{f_Kfp}$, a straightforward calculation shows that the kernel of $\varphi^{f_Kf}-1$  is trivial unless $\delta(\varpi_K)^f=1$, in which case it is $R\otimes k'$.  We now need to compute the kernel of $\varphi^{f_K}-1$ on $D_{f_K,\delta(\varpi_K)}\otimes_{k_K}k'$.  But there is a basis $\{e_0,\ldots,e_{f-1}\}$ of $k'/k_K$ such that $\varphi^{f_K}$ acts via $\varphi^{f_K}(e_i)=e_{i-1}$, where the indices are taken modulo $f$, so the kernel of $\varphi_D^{f_K}-1$ is trivial unless $\delta(\varpi_K)=1$, in which case it is $D_{f_K,\alpha}$.  We have reduced to computing the kernel of $\varphi-1$ on $D_{f_K,1}$, but the construction makes clear that this kernel is precisely $R$.

Now suppose $\delta_1$ is trivial, so that $H_{\varphi,\Gamma_K}^0(\Lambda_{R,\rig,K}(\delta))$ is $R^{\Gamma_K=1}$.  If $\gamma\in\Gamma_K$ is a topological generator of $\Gamma_K$, it acts on $R$ via multiplication by $\beta$ for some $\beta\in\Lambda_{R,(0,b],K}$.  This clearly fixes no elements unless $\beta=1$, in which case it fixes all of $R$.
\end{proof}

\begin{cor}
	If $D$ is a rank-$1$ $(\varphi,\Gamma)$-module over $\Lambda_{R,\rig,K}$ of character type, then $D$ has no proper non-trivial sub-$(\varphi,\Gamma)$-module or quotient $(\varphi,\Gamma)$-module.
\end{cor}

\begin{lemma}\label{lemma: alpha varphi - 1 surj on psi=1}
Suppose $\alpha\in R^\times$ satisfies $v_R(\alpha)< 0$.  Then if $f\in\Lambda_{R,[0,b],\Q_p}$ is in the image of $\varphi$, $f$ is also in the image of $\alpha\varphi-1\colon\Lambda_{R,[0,b],\Q_p}\rightarrow \Lambda_{R,[0,b/p],\Q_p}$.
\end{lemma}
\begin{proof}
	We are looking for a solution to the equation $(\alpha\varphi-1)(g)=\varphi(f')$; applying $\psi$ to both sides, it suffices to show that the sum $\sum_{k\geq 0}(\alpha^{-1}\psi)^k$ converges on $\Lambda_{R,[0,b],K}$.  Since $\Lambda_{R,[0,b],K}=\Lambda_{R_0,[0,b],K}\left[\frac 1 u\right]$ and $\psi$ is $R$-linear, we may assume that $f'\in\Lambda_{R_0,[0,b],\Q_p}$.  If $f'=\sum_{i\in\Z}\alpha_i\overline\pi^i$,  we may write 
\[	\sum_{i\in\Z}\alpha_i\overline\pi^i = \sum_{j=0}^{p-1}\left(\sum_{i\in\Z}\alpha_{pi+j}\overline\pi^{pi}\right)\overline\pi^j = \sum_{j=0}^{p-1}\left(\sum_{i\in\Z}\alpha_{pi+j}\overline\pi^{pi}\right)\left(\sum_{k=0}^j\binom j k (-1)^k\varepsilon^{j-k}\right)	\]
	Since $\sum_{i\in\Z}\alpha_{pi+j}\overline\pi^{pi}$ is in the image of $\varphi$, we see that
	\[	\psi\left(\sum_{i\in\Z}\alpha_i\overline\pi^i\right) = \sum_{i\in\Z}\left(\sum_{j=0}^{p-1}(-1)^j\alpha_{pi+j}\right)\overline\pi^i	\]
	We may write $f'=f_-+f_+$, where 
	\[	f_-=\sum_{i<0}\alpha_i\overline\pi^i\qquad\text{ and }\qquad f_+\coloneqq \sum_{i\geq 0}\alpha_i\overline\pi^i	\]
	By definition $f_+\in \Lambda_{R_0,[0,\infty],\Q_p}\subset\Lambda_{R_0,[0,b],\Q_p,0}$, and we see that $\psi(f_+)$ is another element of $\Lambda_{R_0,[0,\infty],\Q_p}$.  Thus, $\sum_{k\geq 0}(\alpha^{-1}\psi)^k$ applied to $f_+$ converges.

	It remains to show that $\sum_{k\geq 0}(\alpha^{-1}\psi)^k$ converges when applied to $f_-$.  But we may compute
	\begin{align*}
		v_{R,b}\left(\psi\left(\sum_{i<0}\alpha_i\overline\pi^i\right)\right) &= \frac{1}{b} \inf_{i<0}\left\{ v_R\left(\sum_{j=0}^{p-1}(-1)^j\alpha_{pi+j}\right) + \frac{pbi}{p-1}\right\}	\\
	&\geq \frac{1}{b} \inf_{i<0}\left\{v_R(\alpha_i) + \frac{pb}{p-1}\left\lfloor \frac i p\right\rfloor\right\}	\\
	&\geq \frac{1}{b} \inf_{i<0}\left\{v_R(\alpha_k) + \frac{pbi}{p-1} \right\}	\\
	&= v_{R,b}\left(\sum_{i<0}\alpha_i\overline\pi^i\right)
	\end{align*}
	where the second inequality follows because $\left\lfloor\frac i p\right\rfloor \geq i$ for $i<0$. 
	It follows that $\sum_{k\geq 0}(\alpha^{-1}\psi)^k$ converges on all of $\Lambda_{R,[0,b],\Q_p}$.
\end{proof}

\begin{lemma}\label{lemma: alpha varphi - 1 surj on open disk}
Suppose $\alpha\in R^\times$.  Then $\alpha\varphi-1\colon\overline\pi_K\Lambda_{R,(0,\infty],K}\rightarrow \overline\pi_K\Lambda_{R,(0,\infty],K}$ is surjective.  If $\alpha\neq 1$, then $\alpha\varphi-1\colon\Lambda_{R,(0,\infty],K}\rightarrow\Lambda_{R,(0,\infty],K}$ is surjective.
\end{lemma}
\begin{proof}
It suffices to show that $\sum_{k\geq 0}(\alpha\varphi)^k$ converges on $\overline\pi_K\Lambda_{R,[a,\infty],K}$ for all $a>0$.  For any $f=\sum_{i\geq 1}\alpha_i\overline\pi_K^i$, we have
\begin{equation*}
\begin{split}
v_{R,a}'((\alpha\varphi)^k(f)) &= k\cdot v_R(\alpha) + \inf_i\left\{v_R(\alpha_i)+\frac{p^{k+1}ai}{p-1}\right\} 	\\
&\geq k\cdot v_R(\alpha) + a(p^k+\cdots+p) +\inf_i\left\{v_R(\alpha_i)+\frac{pai}{p-1}\right\}	\\
&= v_{R,a}'(f) + k\cdot v_R(\alpha) + a(p^k+\cdots+p)
\end{split}
\end{equation*}
Thus, for any $\alpha\in R^\times$ and any $a>0$, the sum $\sum_{k\geq 0}(\alpha\varphi)^k(f)$ converges to an element of $\overline\pi_K\Lambda_{R,[a,\infty],K}$, as desired.

If $\alpha\neq 1$, then $(\alpha\varphi-1)\left(\frac{1}{\alpha-1}\right)=1$, so $R$ is also in the image of $\alpha\varphi-1$.
\end{proof}

\begin{cor}\label{cor: coker alpha varphi - 1}
Suppose $\alpha\in R^\times$ satisfies $v_R(\alpha)< 0$.  If $f\in\Lambda_{R,(0,b],\Q_p}$, then there is some $g\in \Lambda_{R,(0,b],\Q_p}$ such that $f-(\alpha\varphi-1)g\in\Lambda_{R,[0,b],\Q_p}^{\psi=0}$.
\end{cor}
\begin{proof}
We have exact sequences
\[	0\rightarrow R_0[\![\overline\pi]\!]\left[\frac 1 u\right]\rightarrow \Lambda_{R,(0,\infty],\Q_p}\oplus\Lambda_{R,[0,b],\Q_p}\rightarrow \Lambda_{R,(0,b],\Q_p}\rightarrow 0	\]
for every $b>0$, so we may write $f=f_++f_-$, where $f_+\in \overline\pi\Lambda_{R,(0,\infty],\Q_p}$ and $f_-\in \Lambda_{R,[0,b],\Q_p}$.  Then we can find $g_+\in \overline\pi\Lambda_{R,(0,\infty],\Q_p}$ and $g_-\in \Lambda_{R,[0,b],\Q_p}$ such that $f_+=(\alpha\varphi-1)(g_+)$ (by Lemma~\ref{lemma: alpha varphi - 1 surj on open disk}) and $f_--(\alpha\varphi-1)(g_-)\in\Lambda_{R,[0,b],\Q_p}^{\psi=0}$ (by Lemma~\ref{lemma: alpha varphi - 1 surj on psi=1} applied to $\varphi\psi(f_-)$), so $f-(\alpha\varphi-1)(g_++g_-)\in \Lambda_{R,[0,b],\Q_p}^{\psi=0}$, as desired.
\end{proof}

\begin{cor}\label{cor: h2 char}
	If $\delta\colon\Q_p^\times\rightarrow R^\times$ is a continuous character trivial on $1+p\Z_p\subset\Z_p^\times$, such that $v_R(\delta(p))<0$, then $H_{\varphi,\Gamma_{\Q_p}}^2(\delta)=0$.
\end{cor}
\begin{proof}
	Corollary~\ref{cor: coker alpha varphi - 1} implies that after subtracting an element of the form $(\alpha\varphi-1)(g)$, any cohomology class of $H_{\varphi,\Gamma_{\Q_p}}^2(\delta)$ has a representative $f\in  \Lambda_{R,[0,b],\Q_p}^{\psi=0}$, for sufficiently small $b$.  But if $\gamma$ is a topological generator of $\Gamma_{\Q_p}/\Delta_{\Q_p}$, by ~\cite[Proposition 4.8]{bellovin2020} $\gamma-1$ acts invertibly on $\Lambda_{R,[0,b],\Q_p}^{\psi=0}$, for sufficiently small $b$, and the result follows.
\end{proof}

Now we can compute $H_{\varphi,\Gamma_{\Q_p}}^1(\delta)$ when $v_R(\delta(p))<0$ and $\delta$ is trivial on $1+p\Z_p$:
\begin{lemma}\label{lemma: h1 char}
	If $\delta\colon\Q_p^\times\rightarrow R^\times$ is a character with $v_R(\delta(p))<0$ and $\delta|_{1+p\Z_p}=1$, then $H_{\varphi,\Gamma_{\Q_p}}^1(\delta)$ is $1$-dimensional.  
\end{lemma}
\begin{proof}
	This follows from Corollary~\ref{thm: euler characteristic}, Lemma~\ref{lemma: h0 character}, and Corollary~\ref{cor: h2 char}.
\end{proof}

\subsection{Tate local duality}\label{section: tate local duality}

We now begin proving Tate's local duality theorem (which we stated in general in Theorem~\ref{thm: tate duality}).
\begin{lemma}
	If $\delta\colon\Q_p^\times\rightarrow R^\times$ is a continuous character such that $v_R(\delta(p))<0$ and $\delta|_{1+p\Z_p}$ is trivial, then Tate duality (as stated in ~\ref{thm: tate duality}) holds for $\Lambda_{R,\rig,\Q_p}(\delta)$ and $\Lambda_{R,\rig,\Q_p}(\delta^{-1}\chi_{\mathrm{cyc}})$.
\end{lemma}
\begin{proof}
Corollary~\ref{thm: euler characteristic} implies that $\dim_RH_{\varphi,\Gamma_{\Q_p}}^1(\delta^{-2}\chi_{\mathrm{cyc}})\geq 1$, and so there is a non-split extension of $(\varphi,\Gamma_{\Q_p})$-modules
\begin{equation}\label{eqn: non-split ext of chars}
0\rightarrow \Lambda_{R,\rig,\Q_p}(\delta^{-1}\chi_{\mathrm{cyc}})\rightarrow D\rightarrow \Lambda_{R,\rig,\Q_p}(\delta)\rightarrow 0
\end{equation}
We claim that $D$ is semistable of slope $0$ (in the sense of ~\cite[\textsection 1.4]{kedlaya08}).  Indeed, if $D'\subset D$ is a rank-$1$ submodule, the corresponding homomorphism $D'\rightarrow \Lambda_{R,\rig,\Q_p}(\delta)$ is either an isomorphism or $0$.  The former would contradict the assumption that ~\ref{eqn: non-split ext of chars} is a non-split extension.  But then we must have $D'=\Lambda_{R,\rig,\Q_p}(\delta^{-1}\chi_{\mathrm{cyc}})$, and $v(\delta^{-1}(p)\chi_{\cyc}(p))>0$.

By ~\cite[Theorem 1.7.1]{kedlaya08}, this implies that $D$ is \'etale, and hence (by ~\cite[Proposition 1.5.5]{kedlaya08}) arises from a $(\varphi,\Gamma)$-module over $\Lambda_{R,[0,b],\Q_p}$ for some $b>0$.  Thus, $D$ comes from a Galois representation, so Tate local duality holds for its cohomology.

We have a long exact sequence in cohomology associated to ~\ref{eqn: non-split ext of chars}:
\[
\begin{tikzcd}
	0 \ar[r] & H_{\varphi,\Gamma_{\Q_p}}^0(\delta^{-1}\chi_{\cyc}) \ar[r] & H_{\varphi,\Gamma_{\Q_p}}^0(D) \ar[r] & H_{\varphi,\Gamma_{\Q_p}}^0(\delta) \ar[out=0, in=180, looseness=1, dll] &	\\
& 	H_{\varphi,\Gamma_{\Q_p}}^1(\delta^{-1}\chi_{\cyc}) \ar[r] & H_{\varphi,\Gamma_{\Q_p}}^1(D) \ar[r] & H_{\varphi,\Gamma_{\Q_p}}^1(\delta) \ar[out=0, in=180, looseness=1, dll] &      \\
& 	H_{\varphi,\Gamma_{\Q_p}}^2(\delta^{-1}\chi_{\cyc}) \ar[r] & H_{\varphi,\Gamma_{\Q_p}}^2(D) \ar[r] & H_{\varphi,\Gamma_{\Q_p}}^2(\delta) \ar[r] & 0
\end{tikzcd}
\]
First we observe that $\delta^{-1}\chi_{\mathrm{cyc}}$ and $\delta$ are non-trivial, so by Lemma~\ref{lemma: h0 character} $H_{\varphi,\Gamma_{\Q_p}}^0(\delta^{-1}\chi_{\mathrm{cyc}})=H_{\varphi,\Gamma_{\Q_p}}^0(\delta)=0$.  Hence we also have $H_{\varphi,\Gamma_{\Q_p}}^0(D)=0$, so duality implies that $H_{\varphi,\Gamma_{\Q_p}}^2(D^\vee(\chi_{\mathrm{cyc}}))=0$.

Since $v_R(\delta(p))<0$, we additionally have $H_{\varphi,\Gamma_{\Q_p}}^2(\delta)=0$, by Corollary~\ref{cor: h2 char}.  Again using the vanishing of $H_{\varphi,\Gamma_{\Q_p}}^0(\delta)$, Corollary~\ref{thm: euler characteristic} implies that $H_{\varphi,\Gamma_{\Q_p}}^1(\delta)$ is $1$-dimensional.

If we dualize ~\ref{thm: euler characteristic} and tensor with $\chi_{\mathrm{cyc}}$, we get a second exact sequence
\[	0\rightarrow \Lambda_{R,\rig,\Q_p}(\delta^{-1}\chi_{\mathrm{cyc}})\rightarrow D^\vee(\chi_{\mathrm{cyc}})\rightarrow \Lambda_{R,\rig,\Q_p}(\delta)\rightarrow 0 \]
and its associated long exact sequence in cohomology. 
Then the cup product (as constructed in ~\cite[Definition 2.3.10]{kpx}) gives us a commutative diagram
\begin{equation*}
\resizebox{\displaywidth}{!}{
$
	\begin{tikzcd}[ampersand replacement=\&]
0 \ar[d]\ar[r] \& H_{\varphi,\Gamma_{\Q_p}}^1(\delta^{-1}\chi_{\mathrm{cyc}}) \ar[d]\ar[r] \& H_{\varphi,\Gamma_{\Q_p}}^1(D^\vee(\chi_{\mathrm{cyc}})) \ar[d]\ar[r] \& H_{\varphi,\Gamma_{\Q_p}}^1(\delta) \ar[d]\ar[r] \& H_{\varphi,\Gamma_{\Q_p}}^2(\delta^{-1}\chi_{\mathrm{cyc}}) \ar[d]	\\
 H_{\varphi,\Gamma_{\Q_p}}^2(\delta^{-1}\chi_{\mathrm{cyc}})^\vee \ar[r] \& H_{\varphi,\Gamma_{\Q_p}}^1(\delta)^\vee \ar[r] \& H_{\varphi,\Gamma_{\Q_p}}^1(D)^\vee \ar[r] \& H_{\varphi,\Gamma_{\Q_p}}^1(\delta^{-1}\chi_{\mathrm{cyc}})^\vee \ar[r] \& 0
\end{tikzcd}	$	}
\end{equation*}
Since $H_{\varphi,\Gamma_{\Q_p}}^1(D^\vee(\chi_{\mathrm{cyc}}))\rightarrow H_{\varphi,\Gamma_{\Q_p}}^1(D)^\vee$ is an isomorphism (by the classical theorem), a diagram chase shows that $H_{\varphi,\Gamma_{\Q_p}}^1(\delta^{-1}\chi_{\mathrm{cyc}})\rightarrow H_{\varphi,\Gamma_{\Q_p}}^1(\delta)^\vee$ is injective, so $\dim_RH_{\varphi,\Gamma_{\Q_p}}^1(\delta^{-1}\chi_{\mathrm{cyc}})\leq \dim_RH_{\varphi,\Gamma_{\Q_p}}^1(\delta)=1$.  But Theorem~\ref{thm: euler characteristic} implies that $\dim_RH_{\varphi,\Gamma_{\Q_p}}^1(\delta^{-1}\chi_{\mathrm{cyc}})\geq 1$, so $\dim_RH_{\varphi,\Gamma_{\Q_p}}^1(\delta^{-1}\chi_{\mathrm{cyc}})=1$ and the map $H_{\varphi,\Gamma_{\Q_p}}^1(\delta^{-1}\chi_{\mathrm{cyc}})\rightarrow H_{\varphi,\Gamma_{\Q_p}}^1(\delta)^\vee$ is an isomorphism.
\end{proof}

\begin{thm}\label{thm: tate local duality}
Tate local duality holds for every projective $(\varphi,\Gamma)$-module $D$ over $\Lambda_{R,\rig,K}$.
\end{thm}
\begin{proof}
We may replace $D$ by $\Ind_K^{\Q_p}D$ and treat the case of $(\varphi,\Gamma)$-modules over $\Lambda_{R,\rig,\Q_p}$.  We may also assume that $D$ is pure of slope $s$, and by replacing it with $D^\vee(\chi_{\mathrm{cyc}})$ if necessary, that $s\geq 0$.

If $s=0$, $D$ is \'etale and the result follows from the comparison with Galois cohomology.  Otherwise, we proceed by induction on the degree of $D$, i.e. $\deg(D)\coloneqq (\rk D)s$.  Let $\delta\colon\Q_p^\times\rightarrow R^\times$ be a continuous character with $v_R(\delta(p))=-1$ and $\delta|_{1+p\Z_p}$ trivial.  Since $\dim_RH_{\varphi,\Gamma_{\Q_p}}^1(D(\delta^{-1}))\geq \rk D\geq 1$ by Theorem~\ref{thm: euler characteristic}, there is a non-split extension
\[	0\rightarrow D\rightarrow D'\rightarrow \Lambda_{R,\rig,\Q_p}(\delta)\rightarrow 0	\]
We will prove that Tate local duality holds for $D'$; since it also holds for $\Lambda_{R,\rig,\Q_p}(\delta)$, we may deduce it for $D$.

If $D'$ is pure, the result follows, since $D'$ has degree $\deg D-1$ and slope $(\deg D-1)/(\rk D +1) < s$.  Otherwise, $D'$ has a unique slope filtration $0=D_0\subset D_1\subset\cdots\subset D_k=D'$ by saturated $(\varphi,\Gamma)$-submodules, such that the successive quotients are pure and $\mu(D_1/D_0)<\mu(D_2/D_1)<\cdots<\mu(D_k/D_{k-1})$.  Then $\mu(D_1)\leq \mu(D')<\mu(D)$. 

We have an exact sequence
\[	0\rightarrow D_1\cap D\rightarrow D_1\rightarrow D_1/(D_1\cap D)\rightarrow 0	\]
Since $D$ is pure of positive slope, $D_1\cap D$ also has positive slope.  Since $D_1/(D_1\cap D)$ is the image of $D_1$ in the quotient $D'\rightarrow\Lambda_{R,\rig,\Q_p}(\delta)$, it has slope (and hence degree) either $0$ or $-1$.  Therefore, $\mu(D_1)>0$, as well.

It follows that $\mu(D_i/D_{i-1})>0$ for all $i$.  Moreover, $\deg D'=\sum_i \deg(D_i/D_{i-1})=\sum_i \mu(D_i/D_{i-1})\cdot\rk(D_i/D_{i-1})$, so $\deg(D_i/D_{i-1})<\deg D$ for all $i$.  Then the inductive hypothesis implies that Tate local duality holds for each $D_i/D_{i-1}$, so it holds for $D'$, and we are done.
\end{proof}

Now we can complete the computation of the cohomology of $(\varphi,\Gamma_K)$-modules of character type when the coefficients are a finite extension of $\F_p(\!(u)\!)$.
\begin{cor}\label{cor: coh char type}
Let $\delta\colon K^\times\rightarrow R^\times$ be a continuous character.  Then
\begin{enumerate}
\item	$H_{\varphi,\Gamma_K}^0(\delta)=0$ unless $\delta$ is the trivial character, in which case $H_{\varphi,\Gamma_K}^0(\delta)$ is a $1$-dimensional $R$-vector space.
\item	$H_{\varphi,\Gamma_K}^2(\delta)=0$ unless $\delta=\chi_{\mathrm{cyc}}\circ\Nm_{K/\Q_p}$, in which case $H_{\varphi,\Gamma_K}^2(\delta)$ is a $1$-dimensional $R$-vector space.
\item	$H_{\varphi,\Gamma_K}^1(\delta)$ is an $R$-vector space of dimension $[K:\Q_p]$ unless either $H_{\varphi,\Gamma_K}^0(\delta)\neq 0$ or $H_{\varphi,\Gamma_K}^2(\delta)\neq 0$, in which case it is an $R$-vector space of dimension $[K:\Q_p]+1$.
\end{enumerate}
\end{cor}
In order to handle the case where $K\neq \Q_p$, we use induction and Lemma~\ref{lemma: induction of character} to reduce to the settled case over $\Q_p$.

\section{Triangulations}

\subsection{Classification of rank-$1$ \texorpdfstring{$(\varphi,\Gamma)$}{(𝜑, Γ)}-modules}

In this section, we show that projective rank-$1$ $(\varphi,\Gamma)$-modules over a pseudorigid space $X$ are free locally on $X$, and up to twisting by a line bundle on $X$, are of character type.  Throughout this section, we will assume that all of our $(\varphi,\Gamma)$-modules are projective.
The proof is largely the same as in~\cite[\textsection 6.2]{kpx}.  We first treat the case where the coefficients are a field, where we can exploit the fact that $\Lambda_{R,(0,b],K}$ is B\'ezout, and then deduce the case where the coefficients are artinian by a deformation argument.
\begin{prop}\label{prop: char type artinian}
	Suppose $R$ is an artin local pseudoaffinoid algebra.  If $D$ is a rank-$1$ $(\varphi,\Gamma)$-module over $\Lambda_{R_0,(0,b],K}$, then there is a unique continuous character $\delta\colon K^\times\rightarrow R^\times$ such that $\mathscr{L}\coloneqq H_{\varphi,\Gamma_K}^0(D(\delta^{-1}))$ is free of rank $1$ over $R$ and the natural map $\Lambda_{R,\rig,K}(\delta)\otimes_R\mathscr{L}\rightarrow D$ is an isomorphism.
\begin{enumerate}
\item	$H_{\varphi,\Gamma_K}^1(D(\delta^{-1}))$ is free over $R$ of rank $1+[K:\Q_p]$.
\item	$H_{\varphi,\Gamma_K}^2(D(\delta^{-1}))=0$.
\end{enumerate}
\end{prop}
\begin{proof}
	This proof is nearly identical to the proof of ~\cite[Lemma 6.2.13]{kpx}, so we only give a sketch.

	When $R$ is a field and $D$ has slope $s$, we may choose $\alpha\in R$ with $v_R(\alpha)=s$.  If $\delta\colon\Q_p^\times\rightarrow R^\times$ be the character with $\delta(p)=\alpha$ and $\delta|_{\Z_p^\times}=1$, then $\delta\circ\Nm_{K/\Q_p}$ is a character $K^\times\rightarrow R^\times$ trivial on $\mathscr{O}_K^\times$ and sending a uniformizer of $K$ to $\alpha^f$, and by construction, the associated $(\varphi,\Gamma)$-module $D(\delta\circ\Nm_{K/\Q_p})$ has slope $s$.  Twisting $D$ by its inverse, we reduce to the \'etale case.  But when $D$ is \'etale, $M\coloneqq \left(\widetilde\Lambda_{R,\rig}\otimes D\right)^{\varphi=1}$ is a Galois representation with $D=D_{\rig}(M)$.  Then local class field theory and the construction of $(\varphi,\Gamma)$-modules of character type imply that there is a unique character $\delta\colon K^\times R^\times$ with $D=\Lambda_{R,\rig,K}(\delta)$.  The calculation of cohomology follows from Corollary~\ref{cor: coh char type}.

	In order to bootstrap to the case where $R$ is an artin local ring, we factor the extension $R\twoheadrightarrow R/\mathfrak{m}_R$ as a sequence of small extensions, that is, extensions of the form
	\[      0\rightarrow I\rightarrow R\rightarrow R'\rightarrow 0  \]
	Here $\mathfrak{m}_R\subset R$ is the maximal ideal of $R$ and $I\subset R$ is a principal ideal with $I\mathfrak{m}_R=0$.  Then $R\rightarrow R'$ is a square-zero thickening, and deformation theory (of characters and of $(\varphi,\Gamma)$-modules) implies that if $D$ is a $(\varphi,\Gamma)$-module over $R$ of rank $1$ with $R'\otimes_RD$ of character type, then $D$ is of character type.

	If $D=\Lambda_{R,\rig,K}(\delta)$, then $H_{\varphi,\Gamma_K}^0(D(\delta^{-1}))$ contains $R$, and considerations on lengths of $R$-modules imply that if $H_{\varphi,\Gamma_K}^0(R'\otimes_RD(\delta^{-1}))=R'$, then $H_{\varphi,\Gamma_K}^0(D(\delta^{-1}))=R$.  Moreover, 
	\[	R/\mathfrak{m}_R\otimes_RH_{\varphi,\Gamma_K}^2(D(\delta^{-1}))=H_{\varphi,\Gamma_K}^2(D_{R/\mathfrak{m}_R}(\delta^{-1}))=0	\]
	so $H_{\varphi,\Gamma_K}^2(D(\delta^{-1}))=0$.  Then the base change spectral sequence implies that the formation of $H_{\varphi,\Gamma_K}^1(D(\delta^{-1}))$ commutes with base change on $R$, and the Euler characteristic formula implies that $\dim_{R/\mathfrak{m}_R}H_{\varphi,\Gamma_K}^1(D_{R/\mathfrak{m}_R}(\delta^{-1}))=1+[K:\Q_p]$. Then by Nakayama's lemma, $H_{\varphi,\Gamma_K}^1(D(\delta^{-1}))$ is free of the same rank, and we are done.

\end{proof}

In order to give a classification over a general base, we again follow the strategy of the proof of ~\cite[Theorem 6.2.14]{kpx} and twist our rank-$1$ $(\varphi,\Gamma)$-module by the universal family of characters.  Then we can use the settled case over artin local rings and cohomology and base change to cut out the appropriate character.  The difficulty is in verifying that the slopes and weights of a family of $(\varphi,\Gamma)$-modules over a pseudoaffinoid algebra are bounded; boundedness of the slopes is the essential content of the following proposition, whose proof we do not duplicate in detail.
\begin{prop}\label{prop: psi-1 surj for twist}
	Let $R$ be a pseudoaffinoid algebra with pseudouniformizer $u$, and let $D$ be a $(\varphi,\Gamma_{\Q_p})$-module over $\Lambda_{R,\rig,\Q_p}$.  Then
\begin{enumerate}
	\item	The quotient $D/(\psi-1)$ is a finitely generated $R$-module.
	\item	If $n\in\Z$, let $\delta_n\colon\Q_p^\times\rightarrow R^\times$ be the character trivial on $\Z_p^\times$ which sends $p$ to $u^n$.  Then for all $n\gg0$, the map $\psi-1\colon D(\delta_{-n})\rightarrow D(\delta_{-n})$ is surjective.
\end{enumerate}
\end{prop}
\begin{proof}[Sketch of proof.]
	The proof of ~\cite[Proposition 3.3.2]{kpx} carries over verbatim.  We take a model of $D$ as a finite projective module over $\Lambda_{R,(0,b],\Q_p}$, consider it as a summand of a finite free module $D'$ with basis $\{\mathbf{e}_i\}$, and carefully analyze the actions of $\varphi$ and $\psi$.  We choose an interval $[a,b]\subset\R$ (depending on $\varphi$ and $\psi$ on $D$) and consider the image $D''$ of $\oplus_i\oplus_{j\in[a,b]\cap\Z}R$ in $D$.  Then $D''$ is a finite $R$-module, and it is possible to show that every element of $D$ differs from an element of $D''$ by something in the image of $\psi-1$.
\end{proof}

We can use this to deduce that the weights are bounded:
\begin{cor}\label{cor: weights bdd}
	Let $R$ be a pseudoaffinoid algebra with ring of definition $R_0\subset R$ and pseudouniformizer $u\in R_0$, and let $D$ be a $(\varphi,\Gamma_{\Q_p})$-module over $\Lambda_{R,\rig,\Q_p}$. Let $\delta_{\mathrm{univ}}\colon\Z_p^\times\rightarrow R^\times$ be the universal character on $\widehat{\Z_p^\times}_R$.  Then the support of the cokernel of $\gamma-1\colon D(\delta_{\mathrm{univ}})/(\psi-1)\rightarrow D(\delta_{\mathrm{univ}})/(\psi-1)$ is contained in a quasi-compact subspace of $\{u\neq 0\}\subset \Spa(R_0\htimes\Z_p[\![\Gamma_{\Q_p}]\!])^{\an}$.
\end{cor}
\begin{proof}
	We consider the action of $\gamma-1$ on $D(\delta_{\mathrm{univ}}/(\psi-1)$.  Choose a presentation $R^{\oplus d}\twoheadrightarrow D/(\psi-1)$, and lift the action of $\gamma$ on $D/(\psi-1)$ to a matrix $1+G\in\GL_d(R)$; replacing $\gamma$ with a power if necessary, we may assume that $G=\left(g_{ij}\right)\in u\Mat_d(R)$.  Then if $r\coloneqq \delta_{\mathrm{univ}}^{-1}(\gamma)-1$, the map $\gamma-1\colon D(\delta_{\mathrm{univ}}/(\psi-1)\rightarrow D(\delta_{\mathrm{univ}}/(\psi-1)$ lifts to the matrix $r+G+rG$; if this matrix is invertible, then $\gamma-1$ is surjective.  

But this matrix fails to be invertible only at points where $\det\left(\frac{r}{1+r}+G\right)$ vanishes.  There is a finite extension $R\rightarrow R'$ such that the characteristic polynomial of $G$ splits over $R'$ as $(X-\lambda_1)\cdots(X-\lambda_d)$, and we see that if $\lvert\frac{r}{1+r}\rvert > \lvert\lambda_i\rvert$ for all $i$, then $\gamma-1$ is invertible.  Since $G$ is topologically nilpotent, the $\lambda_i$ are also topologically nilpotent, so we see that there is some $N\gg0$ such that the vanishing locus of $\det\left(\frac{r}{1+r}+G\right)$ is contained in $\{\lvert \left( \frac{r}{1+r} \right)^N\rvert\leq \lvert u\rvert\neq 0\}\subset \Spa(R_0'\htimes \Z_p[\![\Gamma_{\Q_p}]\!])$ (where $R_0'$ is a ring of definition of $R'$).

Thus, we see that the open affinoid subspace $\{\lvert r^N\rvert\leq \lvert u\rvert\neq 0\}\subset \Spa(R_0\htimes\Z_p[\![\Gamma_{\Q_p}]\!])^{\an}$ is the quasi-compact subspace we were looking for.
\end{proof}

Now we give the desired general classification. The primary difference from the argument of ~\cite[Theorem 6.2.14]{kpx} is in using Corollary~\ref{cor: weights bdd} to bound the weight, rather than arguments using Sen weights.
\begin{thm}\label{thm: rank 1 classification}
	Let $X$ be a pseudoaffinoid algebra with pseudouniformizer $u$, and let $D$ be a rank-$1$ $(\varphi,\Gamma)$-module over $\Lambda_{X,\rig,K}$.  Then there exists a unique continuous character $\delta\colon K^\times \rightarrow \Gamma(\O_X,X)^\times$ and a unique invertible sheaf $\mathscr{L}$ on $X$ such that $D\cong \Lambda_{X,\rig,K}(\delta)\otimes_{\O_X}\mathscr{L}$.
\end{thm}
\begin{remark}
If such a $\delta$ and $\mathscr{L}$ exist, then $\mathscr{L}(U)=H_{\varphi,\Gamma_K}^0(D(\delta^{-1})|_U)$ for every open subspace $U\subset X$.
\end{remark}
\begin{proof}
	We may assume $X=\Spa R$ is affinoid.  We first treat uniqueness.  Since the formation of $H_{\varphi,\Gamma_K}^0(D(\delta^{-1}))$ commutes with flat base change on $R$, it suffices to show that if $H_{\varphi,\Gamma_K}^0(\Lambda_{R,\rig,K}(\delta))$ is locally free of rank $1$ over $R$, then $\delta$ is trivial.  There is a Zariski-open dense subspace $U\subset \Spa R$ such that $H_{\varphi,\Gamma_K}^i(\Lambda_{R,\rig,K}(\delta)|_U)$ is flat for all $i$; if $x\in U$ and $\mathfrak{m}_x\subset R$ is the corresponding maximal ideal, then the base change spectral sequence implies that $H_{\varphi,\Gamma_K}^i(R/\mathfrak{m}_x^k\otimes_R\Lambda_{R,\rig,K}(\delta))\cong R/\mathfrak{m}_x^k\otimes_RH_{\varphi,\Gamma_K}^i(\delta)$ for all $i$ and all $k\geq 1$.  In particular, $H_{\varphi,\Gamma_K}^0(\Lambda_{R/\mathfrak{m}_x^k,\rig,K}(\delta))$ is free of rank $1$ over $R/\mathfrak{m}_x^k$, which implies that $\delta\colon K^\times\rightarrow (R/\mathfrak{m}_x^k)^\times$ is trivial for all $k\geq 1$.  It follows that $\delta\colon K^\times\rightarrow (R_{U'})^\times$ is trivial for all affinoid $U'\subset U$.  But the condition $\delta=1$ defines a Zariski-closed subspace of $\Spa R$; since it contains a Zariski-open dense subspace, it is all of $\Spa R$.

To show existence, we follow ~\cite{kpx} and consider the twist of $D$ by the inverse of the universal family of characters $\delta_{\mathrm{univ}}$ over $\widehat{K^\times}_R$; recall that $\widehat{K^\times}_R\coloneqq \G_m^{\ad}\times_{\Z}\Spa\Z_p[\![\O_K^\times]\!]\times_{\Z_p}\Spa R$ is the moduli spaces of continuous characters of $K^\times$ valued in pseudoaffinoid $R$-algebras.

This twist $D(\delta_{\mathrm{univ}}^{-1})$ is a $(\varphi,\Gamma_K)$-module over $\widehat{K^\times}_R$, and we use Tate local duality to cut out a subspace corresponding to the desired character.  More precisely, we let $\Gamma_D'$ and $\Gamma_D''$ be the support of $H_{\varphi,\Gamma_K}^2(D^\vee(\delta_{\mathrm{univ}}\chi_{\mathrm{cyc}}))$ and $H_{\varphi,\Gamma_K}^2(D(\delta_{\mathrm{univ}}^{-1}\chi_{\mathrm{cyc}}))$ in the pseudorigid space
\[	\widehat{K^\times}_R\subset\G_m^{\ad}\times_{\Z}\left( \Spa(R_0\htimes \Z_p[\![\O_K^\times]\!] \right)^{\an}	\]
	respectively, and let $\Gamma_D\coloneqq \Gamma_D'\times_{\widehat{K^\times}_R}\Gamma_D''$.  Since the formation of $H_{\varphi,\Gamma_K}^2$ commutes with arbitrary base change on $\Spa R$, the formation of $\Gamma_D'$ and $\Gamma_D''$, and hence $\Gamma_D$, commutes with arbitrary base change on $\Spa R$.  

	There is a natural projection map $\Gamma_D\rightarrow \Spa R$; a section induces a morphism $\Spa R\rightarrow \widehat{K^\times}_R$, or equivalently, a continuous character $\delta\colon K^\times\rightarrow R^\times$.  We will show that $\Gamma_D\rightarrow \Spa R$ is actually an isomorphism.  

	Granting this, we may replace $D$ with $D(\delta_D^{-1})$, where $\delta_D\colon K^\times\rightarrow R^\times$ is the continuous character corresponding to $\Spa R=\Gamma_D\rightarrow \widehat{K^\times}_R$, so that $\Gamma_D$ corresponds to the trivial character.  Then we need to show that $H_{\varphi,\Gamma_K}^0(D)$ is a line bundle over $\Spa R$, and $D\cong \Lambda_{R,\rig,K}\otimes_RH_{\varphi,\Gamma_K}^0(D)$ as a $(\varphi,\Gamma_K)$-module.  If $R'$ is a pseudoaffinoid artin local ring and $R\rightarrow R'$ is a homomorphism, there is a unique continuous character $\delta'\colon K^\times\rightarrow {R'}^\times$ such that $H_{\varphi,\Gamma_K}^0(D_{R'}({\delta'}^{-1}))$ is free of rank $1$ over $R'$ and the natural map $\Lambda_{R',\rig,K}\otimes_{R'}H_{\varphi,\Gamma_K}^0(D_{R'}({\delta'}^{-1}))\rightarrow D_{R'}({\delta'}^{-1})$ is an isomorphism, and in addition, $H_{\varphi,\Gamma_K}^1(D_{R'}({\delta'}^{-1}))$ is free of rank $1+[K:\Q_p]$ and $H_{\varphi,\Gamma_K}^2(D_{R'}({\delta'}^{-1}))=0$.  

	Thus, the formation of $H_{\varphi,\Gamma_K}^0(D_{R'}({\delta'}^{-1}))$ commutes with arbitrary base change on $R'$; in particular, $H_{\varphi,\Gamma_K}^0(D_{R'/\mathfrak{m}_{R'}}({\delta'}^{-1}))$ is non-zero.  
Since $H_{\varphi,\Gamma_K}^2(D_{R'/\mathfrak{m}_{R'}}^\vee(\delta'\chi_{\mathrm{cyc}}))$ and $H_{\varphi,\Gamma_K}^2(D_{R'/\mathfrak{m}_{R'}}({\delta'}^{-1}\chi_{\mathrm{cyc}}))$ are dual to $H_{\varphi,\Gamma_K}^0(D_{R'/\mathfrak{m}_{R'}}({\delta'}^{-1}))$ and $H_{\varphi,\Gamma_K}^0(D_{R'/\mathfrak{m}_{R'}}^{\vee}(\delta'))$, respectively, and the formation of $H_{\varphi,\Gamma_K}^2$ commutes with arbitrary base change on $R$, we see that $H_{\varphi,\Gamma_K}^2(D_{R'}^\vee(\delta'\chi_{\mathrm{cyc}}))$ and $H_{\varphi,\Gamma_K}^2(D_{R'}({\delta'}^{-1}\chi_{\mathrm{cyc}}))$ are both non-zero.  Thus, the graph of the morphism $\Spa R'\rightarrow \widehat{K^\times}_R$ induced by $\delta'$ is contained in $\Gamma_D$; since $\Gamma_D$ corresponds to the trivial character, $\delta'$ is trivial.

In other words, for any homomorphism $R\rightarrow R'$ with $R'$ a pseudoaffinoid artin local ring, $H_{\varphi,\Gamma_K}^0(D_{R'})$ is free of rank $1$ over $R'$, $H_{\varphi,\Gamma_K}^1(D_{R'})$ is free of rank $1+[K:\Q_p]$, and $H_{\varphi,\Gamma_K}^2(D_{R'})=0$; on residue fields, this implies that $H_{\varphi,\Gamma_K}^2(D_{R'/\mathfrak{m}_{R'}})=0$, so by Nakayama's lemma, $H_{\varphi,\Gamma_K}^2(D_{R'})=0$, as well.  This implies that $H_{\varphi,\Gamma_K}^2(D)$ is locally free of rank $0$, so by the base change spectral sequence, the formation of $H_{\varphi,\Gamma_K}^1(D)$ commutes with arbitrary base change on $R$.  It follows that $H_{\varphi,\Gamma_K}^1(D)$ is locally free of rank $1+[K:\Q_p]$, so the base change spectral sequence again implies that the formation of $H_{\varphi,\Gamma_K}^0(D)$ commutes with arbitrary base change on $R$, and we conclude that $H_{\varphi,\Gamma_K}^0(D)$ is locally free of rank $1$, as desired.

We now prove that $\Gamma_D\rightarrow \Spa R$ is an isomorphism.  In fact, it suffices to prove that $\Gamma_D$ is affinoid: by Theorem~\ref{thm: local isom aff}, an isomorphism of pseudoaffinoid algebras can be detected on the level of completed local rings, and by Proposition ~\ref{prop: char type artinian}, $\Gamma_D\rightarrow \Spa R$ induces an isomorphism on the completed local ring at every maximal point of $\Spa R$.

Since $\Gamma_D$ is a Zariski-closed subspace of the quasi-Stein space $\widehat{K^\times}_R$, it is enough to show that that it is contained in an affinoid subspace. We replace $D$ with $\Ind_K^{\Q_p}D$.  

As in ~\cite[Lemma 6.2.18]{kpx}, we first check that the image of $\Gamma_D$ in $\G_{m,R}$ is bounded.  By Proposition~\ref{prop: psi-1 surj for twist}, there is some $N\geq 0$ such that for all $n\geq N$, 
\[	\psi-1\colon D^\vee(\delta_{-n}\chi_{\mathrm{cyc}})\rightarrow D^\vee(\delta_{-n}\chi_{\mathrm{cyc}})	\]
and 
\[	\psi-1\colon D(\delta_{-n}\chi_{\mathrm{cyc}})\rightarrow D(\delta_{-n}\chi_{\mathrm{cyc}})	\]
are surjective.  Surjectivity is preserved under arbitrary base change $R\rightarrow R'$, and the isomorphism $H_{\varphi,\Gamma_K}^\bullet\rightarrow H_{\psi,\Gamma_K}^\bullet$ from Lemma~\ref{lemma: phi gamma vs psi gamma coh} implies that 
\[	H^2(D^\vee(\delta_{-n}\delta'\chi_{\mathrm{cyc}}))=H_{\varphi,\Gamma_K}^2(D(\delta_{-n}\delta'\chi_{\mathrm{cyc}}))=0	\]
for all continuous characters $\delta'\colon \mathscr{O}_K^\times\rightarrow {R'}^\times$.  Thus, if $T$ denotes the coordinate on $\G_m$, the image of $\Gamma_D'$ is contained in the subspace $\left\{\lvert T\rvert\leq \lvert u^N\rvert\right\}\subset \G_{m,R}$ and the image of $\Gamma_D''$ is contained in the subspace $\left\{\lvert T^{-1}\rvert\leq \lvert u^N\rvert\right\}\subset \G_{m,R}$.

We let $C_{R,N}\coloneqq \{\lvert u^{-N}\rvert\leq \lvert T\rvert \leq \lvert u^N\rvert\}$ denote the annulus above, and we replace $\Spa R$ and $D$ with $C_{R,N}$ and the universal twist of $D$.  Then by Corollary~\ref{cor: weights bdd}, $\Gamma_D$ is contained in a pseudoaffinoid subspace $\{\lvert ([\gamma]-1)^{N'}\rvert\leq \lvert u\rvert\neq 0\}$ for some $N'\gg0$, so we are done.
\end{proof}

\subsection{Interpolating triangulations}

Trianguline $(\varphi,\Gamma)$-modules are those which are extensions of $(\varphi,\Gamma)$-modules of character type.  More precisely,
\begin{definition}
	Let $X$ be a pseudorigid space over $\mathscr{O}_E$ for some finite extension $E/\Q_p$, let $K/\Q_p$ be a finite extension, and let $\underline\delta=(\delta_1,\ldots,\delta_d)\colon (K^\times)^d\rightarrow \Gamma(X,\mathscr{O}_X^\times)$ be a $d$-tuple of continuous characters.  A $(\varphi,\Gamma_{K})$-module $D$ is \emph{trianguline with parameter $\underline\delta$} if (possibly after enlarging $E$) there is an increasing filtration $\Fil^\bullet D$ by $(\varphi,\Gamma_{K})$-modules and a set of line bundles $\mathscr{L}_1,\ldots,\mathscr{L}_d$ such that $\gr^iD\cong \Lambda_{X,\rig,K}(\delta_i)\otimes \mathscr{L}_i$ for all $i$.

	If $X=\Spa R$ where $R$ is a field, we say that $D$ is \emph{strictly trianguline with parameter $\underline\delta$} if for each $i$, $\Fil^{i+1}D$ is the unique sub-$(\varphi,\Gamma_{K})$-module of $D$ containing $\Fil^iD$ such that $\gr^{i+1}D\cong \Lambda_{R,\rig,K}(\delta_{i+1})$.  Equivalently, $D$ is trianguline with parameter $\underline\delta$ and $H_{\varphi,\Gamma}^0\left((\Fil^iD)^\vee(\delta_i)\right)$ is $1$-dimensional.
\end{definition}

We wish to interpolate triangulations at dense sets of points of pseudorigid spaces.
\begin{definition}
	Let $X$ be a pseudorigid space, and let $\Omega\subset X$ be a set of maximal points.  We say that $\Omega$ is \emph{Zariski-dense} if any Zariski-closed subspace $Z\subset X$ containing $\Omega$ also contains the underlying reduced space $X^{\red}$.  We say that a Zariski-dense set $\Omega\subset X$ is \emph{very Zariski dense} if for every $x\in X$ and every irreducible affinoid neighborhood $U\subset X$ containing $x$, $\Omega\cap U$ is Zariski-dense in $U$, that is, any function in $\O_X(U)$ vanishing on all of $\Omega\cap U$ is nilpotent.
	\label{def: zariski very dense}
\end{definition}

\begin{lemma}\label{lemma: quotient interpolates mod p}
	Let $X=\Spa R$ be a reduced pseudorigid space over $\Z_p$ with $p\notin R^\times$, let $D$ be a $(\varphi,\Gamma)$-module over $\Lambda_{R,\rig,K}$, and let $\delta\colon K^\times\rightarrow R^\times$ be a continuous character such that $H_{\varphi,\Gamma_K}^0(D^\vee(\delta))$ is free of rank $1$ over $R$ and $H_{\varphi,\Gamma_K}^i(D^\vee(\delta))$ has Tor-dimension at most $1$ for $i=1,2$.  Then the morphism $D\rightarrow \Lambda_{R,\rig,K}(\delta)$ corresponding to a basis of $H_{\varphi,\Gamma_K}^0(D^\vee(\delta))$ is surjective over an open subspace $U\subset X$ containing $\{p=0\}\subset X$.
\end{lemma}
\begin{proof}
Choose a basis element of $H_{\varphi,\Gamma_K}^0(D^\vee(\delta))$; there is some $b>0$ such that the corresponding homomorphism $D\rightarrow \Lambda_{R,\rig,K}(\delta)$ is defined over $\Lambda_{R,(0,b],K}$, and we may view it as a morphism of coherent sheaves over the corresponding quasi-Stein space.  Moreover, $\varphi$-equivariance means that to check surjectivity, it suffices to check that $\Lambda_{R,[b/p,b],K}\otimes_{\Lambda_{R,(0,b],K}}D\rightarrow \Lambda_{R,[b/p,b],K}\otimes_{\Lambda_{R,(0,b],K}}\Lambda_{R,(0,b],K}(\delta)$ is surjective.

The morphism $\Lambda_{R,[b/p,b],K}\otimes_{\Lambda_{R,(0,b],K}}D\rightarrow \Lambda_{R,[b/p,b],K}\otimes_{\Lambda_{R,(0,b],K}}\Lambda_{R,(0,b],K}(\delta)$ fails to be surjective on a Zariski-closed subspace $Z\subset\Spa \Lambda_{R,[b/p,b],K}$.  Since $\Spa \Lambda_{R,[b/p,b],K}$ is affinoid, so is $Z$.

Consider specializations at the characteristic $p$ maximal points $x\in\Spa R$.
If $H_{\varphi,\Gamma_K}^0(D^\vee(\delta))$ is flat of rank $1$ over $R$, then $k_x\otimes_RH_{\varphi,\Gamma_K}^0(D^\vee(\delta))$ is a $1$-dimensional $k_x$-vector space. If $H_{\varphi,\Gamma_K}^i(D^\vee(\delta))$ has Tor-dimension at most $1$ for $i=1,2$, then the specialization maps $R\rightarrow k_x$ give us exact sequences
\[	0\rightarrow k_x\otimes_RH_{\varphi,\Gamma_K}^0(D^\vee(\delta))\rightarrow H_{\varphi,\Gamma_K}^0(k_x\otimes_RD^\vee(\delta))\rightarrow \Tor_1^R(H_{\varphi,\Gamma_K}^1(D^\vee(\delta)),k_x)\rightarrow 0	\]
by Corollary~\ref{cor: coh and base change}.  
Thus, the induced maps $k_x\otimes_RD\rightarrow k_x\otimes_R\Lambda_{R,\rig,K}(\delta)$ are non-zero, and if $k_x$ has positive characteristic, this implies that the corresponding map is surjective.

Thus, $p$ is a nowhere-vanishing function on $Z$, and since $Z$ is affinoid, the maximum modulus principle discussed in Appendix~\ref{app: max mod principle} implies that $p|_Z$ is bounded away from $0$.  That is, there is some $\lambda$ such that $\{\lvert p\rvert \leq \lambda\}\cap Z$ is empty.  Setting $U\coloneqq  \{\lvert p\rvert \leq \lambda\}\subset X$ yields the desired subspace.
\end{proof}

\begin{thm}\label{thm: interpolate rk 1 quot}
	Let $X$ be a reduced pseudorigid space over $\Z_p$, let $D$ be a $(\varphi,\Gamma_K)$-module over $X$ of rank $d$, and let $\delta\colon K^\times\rightarrow\Gamma(X,\mathscr{O}_X^\times)$ be a continuous character.  Suppose there is a Zariski-dense set $X^{\mathrm{cl}}\subset X$ of maximal points such that for every $x\in X^{\mathrm{cl}}$, $H_{\varphi,\Gamma_K}^0(D_x^\vee(\delta_x))$ is $1$-dimensional and the image of $\Lambda_{k_x,\rig,K}$ under any basis of this space is saturated in $D_x^\vee(\delta_x)$.  Then there exists a proper birational morphism $f\colon X'\rightarrow X$ of reduced pseudorigid spaces, a line bundle $\mathscr{L}$ on $X'$, a homomorphism $\lambda\colon f^\ast D\rightarrow \Lambda_{X',\rig,K}(\delta)\otimes_{X'}\mathscr{L}$ of $(\varphi,\Gamma_K)$-modules, and an open subspace $U\subset X'$ containing $\{p=0\}$ such that
\begin{enumerate}
\item	$\lambda|_U\colon f^\ast D|_U\rightarrow \Lambda_{U,\rig,K}(\delta|_U)\otimes_U\mathscr{L}|_U$ is surjective
\item	the kernel of $\lambda|_U$ is a $(\varphi,\Gamma_K)$-module of rank $d-1$
\end{enumerate}
\end{thm}
\begin{proof}
We may replace $X$ with its normalization (using the theory of normalizations of pseudorigid spaces developed in~\cite{johansson-newton17}), and we may consider the connected components of $X$ separately.  

Using perfectness of $C_{\varphi,\Gamma_K}^\bullet(D^\vee(\delta))$, we may use ~\cite[Corollary 6.3.6(2)]{kpx} to construct a proper birational morphism $f_0\colon X'\rightarrow X$ such that $D'\coloneqq f_0^\ast(D^\vee(\delta))$ has $H_{\varphi,\Gamma_K}^0(D')$ flat and $H_{\varphi,\Gamma_K}^i(D')$ with Tor-dimension at most $1$ for $i=1,2$.  Then for any maximal point $x\in X'$, the base change spectral sequence gives us a short exact sequence
\[      0\rightarrow k_x\otimes_RH_{\varphi,\Gamma_K}^0(D')\rightarrow H_{\varphi,\Gamma_K}^0(k_x\otimes_RD')\rightarrow \Tor_1^R(H_{\varphi,\Gamma_K}^1(D'),k_x)\rightarrow 0       \]
By ~\cite[Lemma 6.3.7]{kpx}, the set of points $x\in X'$ such that the last term is non-zero is a Zariski-closed subspace $Z_0'\subset X'$ whose complement is open and dense.  Thus, $H_{\varphi,\Gamma_K}^0(D')$ is flat of rank $1$.  Letting $\mathscr{L}\coloneqq H_{\varphi,\Gamma_K}^0(D')^\vee$, we obtain a homomorphism $\lambda_0\colon f^\ast D\rightarrow \Lambda_{X',\rig,K}(\delta)\otimes_{X'}\mathscr{L}$.

The formation of $H_{\varphi,\Gamma_K}^0(D')$ commutes with flat base change on $X$; we may find a collection $\{X_i'\}$ of open pseudoaffinoid subspaces of $X'$ such that $H_{\varphi,\Gamma_K}^0(D')|_{X_i'}$ is free, $\{p=0\}\subset \cup_i X_i'$, and $p$ is not invertible on $X_i'$.  Then we may apply Lemma~\ref{lemma: quotient interpolates mod p} to conclude that $\lambda_0|_{X_i'}$ is surjective (possibly after shrinking $X_i'$).  Setting $U\coloneqq \cup X_i'$, we see that $X'$, $U\subset X'$, and $\lambda_0$ satisfy the first of our desired properties.

To check the second claim, observe that for some $b>0$ we have an exact sequence over $U$
\[	0\rightarrow P\rightarrow \Lambda_{U,(0,b],K}\otimes D'|_U\rightarrow \Lambda_{U,(0,b],K}(\delta)\otimes_{U}\mathscr{L}|_U\rightarrow 0	\]
Since $\Lambda_{U,(0,b],K}(\delta)\otimes_{X'}\mathscr{L}$ is $R'$-flat, this sequence remains exact after specializing at any point $x\in U$, so $k_x\otimes P$ is a $(\varphi,\Gamma_K)$-module of rank $d-1$.  It follows by ~\cite[Lemma 2.1.8(1)]{kpx} that $P$ is a vector bundle of rank $d-1$ over the quasi-Stein space associated to $\Lambda_{U,(0,b],K}$, and hence is a $(\varphi,\Gamma_K)$-module of the correct rank.
\end{proof}

\begin{remark}
	The morphism $f\colon X'\rightarrow X$ is, in general, not compatible with the analogous morphism constructed in ~\cite[Theorem 6.3.9]{kpx}; in that argument, the authors make an additional blow-up, in order to control the cohomology groups of $f^\ast M/t$, which is what permits them to deduce that $X^{\mathrm{cl}}\subset f^{-1}(U)$.  But Fontaine's element $t$ does not make sense in our mixed- or positive-characteristic overconvergent period rings, so we cannot deduce that $X^{\mathrm{cl}}\subset f^{-1}(U)$.
\end{remark}

As in ~\cite[Corollary 6.3.10]{kpx}, we may deduce the following:
\begin{cor}\label{cor: interpolate triangulation}
	Let $X$ be a reduced pseudorigid space over $\Z_p$, all of whose connected components are irreducible.  Let $M$ be a $(\varphi,\Gamma_K)$-module over $X$ of rank $d$ and let $\underline\delta\coloneqq (\delta_1,\ldots,\delta_d)\colon (K^\times)^d\rightarrow\Gamma(X,\mathscr{O}_X^\times)$ be a parameter such that $D|_x$ is strictly trianguline with parameter $\underline\delta$ at a Zariski dense set $X^{\mathrm{cl}}\subset X$ of maximal points $x\in X$.  Then there exists a proper birational morphism $f\colon X'\rightarrow X$ of reduced pseudorigid spaces, an increasing filtration $\Fil^\bullet(f^\ast D)$, and an open subspace $U\subset X'$ containing $\{p=0\}$ such that
	\begin{enumerate}
		\item	$\left(\Fil^\bullet (f^\ast D)\right)|_x$ is a strictly trianguline filtration on $(f^\ast D)|_x$ for all $x\in U$,
		\item	there are line bundles $\mathscr{L}_i$ on $U$ and isomorphisms of $(\varphi,\Gamma_K)$-modules $\gr^i(f^\ast D|_U)\rightarrow \Lambda_{U,\rig,K}(\delta_i)\otimes_{U}\mathscr{L}_i$
	\end{enumerate}
\end{cor}
\begin{proof}
	We may assume that $X$ is affinoid.  Then for any $\lambda\in \R_+$, setting $U_\lambda:=\{\lvert p\rvert_x\leq \lambda\}$ and $V_{\lambda}:=\{\lvert p\rvert_x\geq \lambda\}$ yields a cover of $X$.  Then at least one of the following holds: $X^{\mathrm{cl}}\cap U_\lambda$ is Zariski-dense in $U_\lambda$, or $X^{\mathrm{cl}}\cap V_\lambda$ is Zariski-dense in $V_\lambda$.  Moreover, if $\lambda'<\lambda$ and $X^{\mathrm{cl}}\cap V_\lambda$ is Zariski-dense in $V_\lambda$, then $X^{\mathrm{cl}}\cap V_{\lambda'}$ is dense in $V_{\lambda'}$.

	Thus, we see that if $X^{\mathrm{cl}}$ doesn't accumulate at $\{p=0\}$ (in the sense of being Zariski-dense in each $U_{\lambda}$), it provides a Zariski-dense subset of $X\smallsetminus \{p=0\}$.  In the latter case, we may apply ~\cite[Corollary 6.3.10]{kpx} to see that over a Zariski-open and dense subspace $W$ of $X\smallsetminus \{p=0\}$, $M|_W$ is trianguline with parameter $\underline\delta$.  Since $U_\lambda\cap W$ is Zariski-open and dense in $U_\lambda$, we see that each $U_\lambda$ contains a Zariski-dense set of points at which $M$ is trianguline with parameter $\underline\delta$.

	Now we may apply Theorem~\ref{thm: interpolate rk 1 quot} inductively to construct $f:X'\rightarrow X$, $U\subset X'$, and $\{\mathscr{L}_i\}$ satisfying the given properties.

\end{proof}

\section{Applications to eigenvarieties}

\subsection{Set-up}
Extended eigenvarieties have been constructed by ~\cite{aip2018}, ~\cite{johansson-newton}, and ~\cite{gulotta} for various groups; these extended eigenvarieties are expected to (and in some cases known to) carry families of Galois representations such that local Galois-theoretic data matches certain Hecke-theoretic data.  At places away from $p$ and the level, this compatibility specifies that the local Galois representation is unramified and gives a characteristic polynomial for Frobenius.  At places dividing $p$, this compatibility specifies that the local Galois representation is trianguline and gives the parameters of the triangulation.

In this subsection, we use our results on trianguline $(\varphi,\Gamma)$-modules to study extended eigenvarieties at the boundary of weight space, in order to address two questions:
\begin{enumerate}
	\item	Are irreducible components proper at the boundary of weight space?
	\item	Are Galois representations at characteristic $p$ points trianguline at $p$?
\end{enumerate}
We will give partial affirmative answers to both questions.

Before stating our assumptions more precisely, we recall the construction of ~\cite{johansson-newton}.  Let $F$ be a number field, let $\mathbf{H}$ be a reductive group over $F$ split at all places above $p$, and set $\G\coloneqq \Res_{F/\Q}\mathbf{H}$.  Fix a tame level by choosing a compact open subgroup $K_\ell\subset \G(\Q_\ell)$ for each prime $\ell\neq p$, such that $K_\ell$ is hyperspecial for all but finitely many $\ell$, and let $I\subset \G(\Q_p)$ be an Iwahori subgroup.  Let $S'$ denote the set of places $w$ of $\Q$ such that either $w=\infty$, or $K_w$ is not hyperspecial, and let $S$ denote the set of places of $F$ lying above the places in $S'$.  Then ~\cite{johansson-newton} proved the following:
\begin{thm}{\cite[Theorems A and B]{johansson-newton}}\label{thm: galois determinants}
	The eigenvarieties for $\G$ constructed in ~\cite{hansen} naturally extend to pseudorigid spaces $\mathscr{X}_{\G}$ equipped with a weight map $\mathrm{wt}\colon \mathscr{X}_{\G}\rightarrow\mathcal{W}$ to extended weight space $\mathcal{W}\coloneqq \left(\Spa \Z_p[\![T_0']\!]\right)^{\an}$, where $T_0'$ is a certain quotient of the $\Z_p$-points of a (split) maximal torus of a model of $\G$ over $\Z_p$.  Moreover, if $F$ is totally real or CM and $\H=\GL_d$, there is a continuous $d$-dimensional determinant $D\colon \O(\mathscr{X}_{\G})[\Gal_{F,S}]\rightarrow \mathscr{O}^+(\mathscr{X}_{\G}^{\red})$ such that $D(1-X\cdot{\Frob}_v) = P_v(X)$ for all $v\not\in S$, where $P_v(X)$ is the Hecke polynomial.
\end{thm}
When $F$ is totally real with $p$ completely split, and $H=\GL_2$, the characteristic $0$ eigenvariety $\mathscr{X}_{\G}^{\rig}$ contains a very Zariski dense set of ``essentially classical'' points (in the sense of ~\cite{chenevier2011}, using ~\cite[Lemme 6.2.10]{chenevier2004}, ~\cite[Lemme 6.2.8]{chenevier2004}, and a ``small slope implies classical'' criterion).  Furthermore, local-global compatibility at places dividing $p$ is known for classical Hilbert modular forms of motivic weight by ~\cite{skinner2009}, ~\cite{liu2012}, ~\cite{saito2009}, ~\cite{blasius-rogawski}, and so in this case $\mathscr{X}_{\G}^{\rig}$ contains a very Zariski dense set of points at which $D$ corresponds to a trianguline Galois representation.

When $H$ is a totally definite quaternion algebra over a totally real field, split at $p$, a similar argument shows that $\mathscr{X}_{\G}^{\rig}$ contains a very Zariski dense set of essentially classical points.  Moreover, the $p$-adic Jacquet--Langlands correspondance of ~\cite{chenevier2005}, ~\cite{birkbeck} can be extended to the pseudorigid setting.  This identifies each irreducible component of a quaternionic eigenvariety with an irreducible component of an eigenvariety for Hilbert modular forms; it follows that a Galois determinant can be pulled back to $\mathscr{X}_{\G}$, and it corresponds to a trianguline representation at a very Zariski dense set of points of $\mathscr{X}_{\G}^{\rig}$.

There is a similar story when $\G$ is a definite unitary group over $\Q$ split at $p$.  Characteristic $0$ eigenvarieties have been constructed ~\cite{chenevier2004}, ~\cite{bellaiche-chenevier} which interpolate classical automorphic forms and carry a family of Galois determinants:
\begin{thm}{\cite[Chapter 7]{bellaiche-chenevier}}
	Let $F/\Q$ be an imaginary quadratic field and let $\G$ be a definite unitary group associated to $F$, split at $p$.  Then the characteristic $0$ eigenvariety $\mathscr{X}_{\G}^{\rig}$ contains a very Zariski dense set of classical points (corresponding to $p$-refined automorphic representations), and there is a continuous determinant $D\colon \O^+(\mathscr{X}_{\G}^{\rig})[\Gal_{F,S}]\rightarrow \O^+(\mathscr{X}_{\G}^{\rig})$ such that $D(1-X\cdot{\Frob}_v) = P_v(X)$ for all $v\not\in S$, where $P_v(X)$ is the Hecke polynomial.
\end{thm}
Moreover, the corresponding Galois representation is known to be trianguline at classical points
; thus, there is a continuous Galois determinant $\overline D\colon \O^+(\overline{\mathscr{X}_{\G}^{\rig}})[\Gal_{F,S}]\rightarrow \O^+(\overline{\mathscr{X}_{\G}^{\rig}})$ defined on the closure of $\mathscr{X}_{\G}^{\rig}$ in $\mathscr{X}_{\G}$, and it corresponds to a trianguline Galois representation at a very Zariski dense set of points.

We will make more precise what kind of trianguline conditions we have (or hope for) at places dividing $p$.  If $\mathbf{T}$ is a split maximal torus of a model of $\G$ over $\Z_p$, consider a splitting of the inclusion $\mathbf{T}(\Z_p)\hookrightarrow \mathbf{T}(\Q_p)$, and let $\Sigma$ denote the kernel.  There are two submonoids $\Sigma^{\mathrm{cpt}}\subset \Sigma^+\subset\Sigma$; we refer the reader to ~\cite[\textsection 3.3]{johansson-newton} for precise definitions, but we note that when $\G(\Z_p)\cong\prod_{v\mid p}\GL_d(\mathscr{O}_{F_v})$, we may take $\mathbf{T}$ to be the standard torus and
\begin{align*}
	\Sigma &= \prod_{v\mid p}\{\mathrm{diag}(\varpi_v^{a_1},\ldots,\varpi_v^{a_d})\mid a_i\in\Z \}	\\
	\Sigma^+ &= \prod_{v\mid p}\{\mathrm{diag}(\varpi_v^{a_1},\ldots,\varpi_v^{a_d})\mid a_{i+1}\geq a_i\}	\\
		\Sigma^{\mathrm{cpt}} &= \prod_{v\mid p}\{\mathrm{diag}(\varpi_v^{a_1},\ldots,\varpi_v^{a_d})\mid a_{i+1}>a_i\}
\end{align*}

The construction of $\mathscr{X}_{\G}$ depends on a choice of $t\in \Sigma^{\mathrm{cpt}}$, which in the above case we take to be $\prod_{v\mid p}\mathrm{diag}(1,\ldots,\varpi_v^{d-1})$; the authors construct a spectral variety $\mathscr{Z}\subset\G_{m,\mathcal{W}}$ using the Fredholm series of the corresponding controlling Hecke operator $U_t\coloneqq [ItI]$, and then construct $\mathscr{X}_{\G}\rightarrow \mathscr{Z}$ finite, such that there is a homomorphism $\psi\colon \mathbf{T}(\Delta^p,K^p)\rightarrow\mathscr{O}(\mathscr{X}_{\G})$.  Here $\mathbf{T}(\Delta^p,K^p)$ is a Hecke algebra with no Hecke operators at places above $p$.

	However, it is possible to make the same construction using other choices of Hecke algebras, and we will need to do so (this is discussed in greater detail in ~\cite[\textsection 3.4]{johansson-newton17}.  In particular, let $\mathscr{A}_p^+\subset \Z_p[\G(\Q_p)//I]$ be the subring generated by the characteristic functions $\mathbf{1}_{[IsI]}$ for $s\in\Sigma^+$.  Then there is an extended eigenvariety $\mathscr{X}_{\G}^{\mathscr{A}_p^+}$ equipped with a homomorphism $\mathbf{T}(\Delta^p,K^p)\otimes_{\Z_p}\mathscr{A}_p^+\rightarrow \mathscr{O}(\mathscr{X}_{\G}^{\mathscr{A}_p^+})$ and a finite morphism $\mathscr{X}_{\G}^{\mathscr{A}_p^+}\rightarrow \mathscr{X}_{\G}$.  There is a surjective finite map $\mathscr{X}_{\G}^{\mathscr{A}_p^+}\rightarrow \mathscr{X}_{\G}$, and we obtain a Galois determinant $\mathscr{O}(\mathscr{X}_{\G}^{\mathscr{A}_p^+,\red})[\Gal_{F,S}]\rightarrow \mathscr{O}(\mathscr{X}_{\G}^{\mathscr{A}_p^+,\red})$ by pulling back the determinant on $\mathscr{X}_{\G}^{\red}$.

		We have finite morphisms $\mathscr{X}_{\G}^{\mathscr{A}_p^+}\rightarrow \mathscr{X}_{\G}\rightarrow\G_{m,\mathcal{W}}$.  By ~\cite[Lemma 3.4.1]{johansson-newton17}, the image of $[IsI]$ in $\mathscr{O}(\mathscr{X}_{\G}^{\mathscr{A}_p^+})$ is invertible for all $s\in\Sigma^+$, and so for $s\in\Sigma$, we can write $s=s'{s''}^{-1}$ for $s', s''\in\Sigma^+$ and obtain $\psi([Is'I])\psi([Is''I]^{-1})\in\mathscr{O}(\mathscr{X}_{\G}^{\mathscr{A}_p^+})^\times$.  Thus, we have a morphism $\mathscr{X}_{\G}^{\mathscr{A}_p^+}\rightarrow \widehat\Sigma_{\mathcal{W}}\coloneqq \Hom(\Sigma,\G_{m,\mathcal{W}})$ such that the diagram
\[
	\begin{tikzcd}
		\mathscr{X}_{\G}^{\mathscr{A}_p^+} \ar[r]\ar[d] & \widehat\Sigma_{\mathcal{W}}\ar[d] \ar[d]      \\
		\mathscr{X}_{\G}\ar[r] & \G_{m,\mathcal{W}}        \\
	\end{tikzcd}
\]
commutes and has finite horizontal maps.  Here the right vertical map is induced by evaluation at $U_t$.  Any choice of a basis of $\Sigma$ will give us parameters $\delta_{i,v}\colon F_v^\times\rightarrow \mathscr{O}(\mathscr{X}_{\G}^{\mathscr{A}_p^+})^\times$.

When $F$ is a number field and $\G=\Res_{F/\Q}\GL_d$ with the standard maximal torus, there is a natural ordered basis $\{s_{i,v}\}_{i,v}$ of $\Sigma$, namely $s_{i,v}\coloneqq \mathrm{diag}(1,\ldots,\varpi_v,\ldots,1)$, where $v|p$ and $\varpi_v$ is placed in the $d-i+1$ slot.  Then (restricting to non-critical points for simplicity), ~\cite[Conjecture 1.2.2(iii)]{hansen} predicts that if $x\in\mathscr{X}_{\G}^{\mathscr{A}_p^+}$ is non-critical, the $(\varphi,\Gamma)$-module corresponding to $D_x|_{\Gal_{F_v}}$ is trianguline with parameter $\underline{\delta_{v}}$ such that  $\delta_{i,v}(\varpi_v)=\psi(s_{i,v})$ (and $\delta_{i,v}|_{\O_{F_v}^\times}$ corresponds to the automorphic weight).  Similarly, in the unitary case sketched above, ~\cite[Proposition 7.5.13]{bellaiche-chenevier} implies that at non-critical points $x\in\mathscr{X}_{\G}^{\mathscr{A}_p^+}$, the $(\varphi,\Gamma)$-module corresponding to $D_x|_{\Gal_{F_v}}$ is trianguline with parameter $\underline{\delta_{v}}$, where $\delta_{i,v}(p)=\psi(\mathrm{diag}(1,\ldots,p,\ldots,1))$.

In the Hilbert, quaternionic, and unitary cases sketched above, we can actually say that at a very Zariski-dense set of classical points, the corresponding $(\varphi,\Gamma)$-module is \emph{strictly} trianguline.

\subsection{Properness at the boundary}

We follow the strategy of ~\cite{diao-liu} to show extended eigenvarieties are proper at the boundary.  We assume we have a Galois representation, and sufficiently many classical points where it is known to be trianguline at $p$, with parameters compatible with the Hecke algebra at $p$:
\begin{thm}\label{thm: proper at boundary}
	Suppose $\mathscr{X}_{\G}^{\mathscr{A}_p^+}$ is an extended eigenvariety such that there is a continuous determinant $D\colon \mathscr{O}(\mathscr{X}_{\G}^{\mathscr{A}_p^+\red})[\Gal_{F,S}]\rightarrow \mathscr{O}(\mathscr{X}_{\G}^{\mathscr{A}_p^+,\red})$ for some number field $F$.  Suppose we have a commutative diagram
	\[
		\begin{tikzcd}
			\Spa R\smallsetminus Z \arrow{r}\arrow{d} & \mathscr{X}_{\G}^{\mathscr{A}_p^+,\red} \arrow{d}	\\
			\Spa R \arrow{r}{\kappa}\arrow[dashed]{ur} & \mathcal{W}
		\end{tikzcd}
	\]
	where $R$ is a normal pseudoaffinoid $\Z_p$-algebra, and $Z\subset \Spa R$ is a Zariski-closed subspace of $\Spa R$ with codimension at least $1$.  Here $\Spa R\rightarrow\mathcal{W}$ corresponds to a weight $\kappa\colon \mathbf{T}(\Z_p)\rightarrow R^\times$.  Suppose in addition that there is an ordered basis $\{s_{i,v}\}$ of $\Sigma$ with $\prod_{j=1}^is_{j,v}\in\Sigma^+$ such that for a Zariski-dense set of points $X^{\mathrm{cl}}\subset \Spa R$ the Galois representation attached to the pull-back of $D$ is strictly trianguline at all places $v\mid p$, with parameters $\{\underline{\delta_{v}}\}$ induced by $\{s_{i,v}\}$ and $\kappa$.  Then the dashed arrow can be filled in.
\end{thm}
\begin{remark}
	For the eigenvarieties discussed in the previous section, we will be able to check that $X^{\mathrm{cl}}$ exists so long as $\Spa R$ has the same dimension as its image in weight space, and classical weights are dense in $\kappa(\Spa R)$.  We can view this as saying that the limit of a family of overconvergent automorphic forms exists, so long as it tends to the deleted subspace in a sufficiently regular manner.  In particular, we may deduce that every irreducible component of the extended eigencurve is proper at the boundary of weight space.
\end{remark}

We first treat the case of a finite morphism:
\begin{lemma}\label{lemma: finite proper}
	Suppose $\mathscr{X}\rightarrow\mathscr{Y}$ is a finite morphism of pseudorigid spaces and we have a commutative diagram
	\[
		\begin{tikzcd}
			\Spa R\smallsetminus Z \arrow{r}\arrow{d} & \mathscr{X} \arrow{d}	\\
			\Spa R \arrow{r}\arrow[dashed]{ur} & \mathscr{Y}
		\end{tikzcd}
	\]
	where $R$ is a normal pseudoaffinoid $\Z_p$-algebra and $Z\subset \Spa R$ is a Zariski-closed subspace of codimension at least $1$.  Then the dashed arrow can be filled in uniquely.
\end{lemma}
\begin{proof}
	We use the Hebbarkeitss\"atze of ~\cite{lourenco}.  Let $u\in R$ be a pseudouniformizer, and suppose $Z$ is defined by the vanishing of the ideal $I=(f_1,\ldots,f_r)\subset R$.  By ~\cite[1.4.4]{huber2013}, the pre-image in $\mathscr{X}$ of any affinoid subspace of $\mathscr{Y}$ is itself affinoid, so we may assume that $\mathscr{X}$ and $\mathscr{Y}$ are both affinoid.  If $\mathscr{X}=\Spa A$ and $\mathscr{Y}=\Spa B$, we choose a pseudouniformizer of $B$; its image $u\in R$ and its image in $A$ are also pseudouniformizers.  

	The morphism $\Spa R\smallsetminus Z\rightarrow \mathscr{X}$ is induced by a compatible sequence of continuous homomorphisms 
	\[	A\rightarrow R_0\left\langle\frac{u^k}{f_1}\right\rangle\left[\frac 1 u\right] \cap \ldots \cap R_0\left\langle\frac{u^k}{f_r}\right\rangle\left[\frac 1 u\right]	\]
	for some noetherian ring of definition $R_0\subset R$ and $k\geq 1$.  Since $Z$ has codimension at least $1$ in $\Spa R$, \cite[Theorem 5.1]{lourenco} implies that the power-bounded functions on $\Spa R\smallsetminus Z$ are precisely $R^\circ$, and we have a continuous homomorphism $A^\circ\rightarrow R^\circ$.  Since $A^\circ$ contains a ring of definition of $A$, we obtain a continuous homomorphism $A\rightarrow R$, as well.  Since the composition $\Spa R\rightarrow \mathscr{X}\rightarrow \mathscr{Y}$ agrees with the specified morphism $\Spa R\rightarrow \mathscr{Y}$ after restricting to $\Spa R\smallsetminus Z$ and $\mathscr{Y}$ is separated, it agrees on all of $\Spa R$.
\end{proof}

Combined with the theory of determinants discussed in Appendix~\ref{appendix: determinants}, the assumption that a determinant $D\colon \mathscr{O}(\mathscr{X}_{\G})[\Gal_{F,S}]\rightarrow \mathscr{O}(\mathscr{X}_{\G})$ exists implies that there is a natural map $\mathscr{X}_{\G}^{\red}\rightarrow X_p$, where $X_p$ is the adic space associated to the deformation rings of all of the determinants attached to isomorphism classes of $d$-dimensional modular residual representations of $\Gal_{F,S}$.

\begin{lemma}\label{lemma: extend determinants}
	Let $R$ be an integral normal pseudoaffinoid $\Z_p$-algebra, and let $Z\subset \Spa R$ be a Zariski-closed subspace of codimension at least $1$.  Then any morphism $\Spa R\smallsetminus Z\rightarrow X_p$ extends uniquely to a morphism $\Spa R\rightarrow X_p$.
	\end{lemma}
	\begin{proof}
		We again use the Hebbarkeitss\"atze of ~\cite{lourenco}.  The pseudorigid space $\Spa R\smallsetminus Z$ is connected, so its image in $X_p$ has constant residual determinant $\overline D$.  Thus, the morphism $\Spa R\smallsetminus Z\rightarrow X_p$ is induced by a homomorphism
		\[	R_{\overline D}\rightarrow \O_{\Spa R}^+\left(\Spa R\smallsetminus Z\right)	\]
		where $R_{\overline D}$ is the pseudodeformation ring parametrizing lifts of $\overline D$.  Since $Z$ has codimension at least $1$ in $\Spa R$, by ~\cite[Theorem 5.1]{lourenco} we have $\O_{\Spa R}^+\left(\Spa R\smallsetminus Z\right) = R^\circ$, so we get a continuous homomorphism $R_{\overline D}\rightarrow R^\circ$ and a morphism $\Spa R\rightarrow \Spa R_{\overline D}$.
	\end{proof}

	\begin{lemma}\label{lemma: sigma^+ power bdd}
		For any $s\in\Sigma^+$, $\psi(s)\in\mathscr{O}(\mathscr{X}_{\G}^{\mathscr{A}_p^+,\red})$ is power-bounded.
	\end{lemma}
	\begin{proof}
	This follows from the construction of ~\cite{johansson-newton} and we use the notation of that paper freely.  By ~\cite[Corollary 3.3.10]{johansson-newton}, the action of $s$ on $\mathcal{D}_\kappa^r$ is norm-decreasing, for any weight $\kappa\colon T_0\rightarrow R^\times$ and any $r\gg 1/p$ (depending on $\kappa$).  It follows that the action of $s$ is power-bounded, hence power-bounded on $C^\ast(K,\mathcal{D}_\kappa^r)$, hence power-bounded on $\mathscr{K}^\bullet\coloneqq \ker^\bullet Q^\ast(U_t)$, and hence power-bounded on $H^\ast(\mathscr{K}^\bullet)$.
	\end{proof}

	Now we are in a position to prove Theorem~\ref{thm: proper at boundary}:
	\begin{proof}[Proof of Theorem~\ref{thm: proper at boundary}]
		Since $\mathscr{X}_{\G}^{\mathscr{A}_p^+,\red}\rightarrow\widehat\Sigma_{\mathcal{W}}$ is finite,  Lemma~\ref{lemma: finite proper} implies that it suffices to lift $\kappa$ to a morphism $\widetilde\kappa\colon \Spa R\rightarrow \widehat\Sigma_{\mathcal{W}}$ (compatibly with the given map $\Spa R\smallsetminus Z\rightarrow \widehat\Sigma_{\mathcal{W}}$).  In other words, we need to show that the image of $\psi(s_{i,v})$ in $\O_{\Spa R}\left(\Spa R\smallsetminus Z\right)$ is an element of $R^\times$ for all $i$ and all $v\mid p$.  

		By the construction of eigenvarieties, the image of $\psi(s_{i,v})$ is a unit of $\O_{\Spa R}\left(\Spa R\smallsetminus Z\right)$ for all $i$ and $v$. Set $U_{i,v}\coloneqq \prod_{j=1}^is_{i,v}\in\Sigma^+$; then $\psi(U_{i,v})$ is a unit of $\O_{\Spa R}\left(\Spa R\smallsetminus Z\right)$, and by Lemma~\ref{lemma: sigma^+ power bdd}, $\psi(U_{i,v})$ is power-bounded for all $i$ and all $v\mid p$.  By ~\cite[Theorem 5.1]{lourenco}, the image of $\psi(U_{i,v})$ in $\O_{\Spa R}\left(\Spa R\smallsetminus Z\right)$ lands in $R^\circ\subset R$, so it remains to see that $\psi(U_{i,v})$ does not vanish at any point of $\Spa R$.

		We use the family of Galois representations on $\mathscr{X}_{\G}^{\mathscr{A}_p^+,\red}$ to prove this.  By Lemma~\ref{lemma: extend determinants}, the determinant $D\colon \mathscr{O}(\mathscr{X}_{\G}^{\mathscr{A}_p^+,\red})[\Gal_{F,S}]\rightarrow \mathscr{O}(\mathscr{X}_{\G}^{\mathscr{A}_p^+,\red})$ extends to a determinant $R[\Gal_{F,S}]\rightarrow R$.   By Lemma~\ref{lemma: det rep after blowup} there is a morphism $f\colon X'\rightarrow \Spa R$ and a family of rank-$d$ Galois representations $M'$ over $X'$ such that $M'$ induces the pullback $f^\ast D_R$ of $D_R$ to $X'$, and we may assume that $X'\rightarrow \Spa R$ is the composition of a blow-up and a finite surjective morphism.  Let $M_v'$ denote the restriction of $M'$ to the local Galois group at $v\mid p$.  

		For each $i=1,\ldots, d$, the exterior product $\wedge^iM_v'$ is a Galois representation. For each $x'$ lying above a point of $X^{\mathrm{cl}}$ the associated $(\varphi,\Gamma)$-module $k_x\otimes D_{\rig}(\wedge^iM_v')$ is strictly trianguline, and the first step in the filtration has character $f^\#\prod_{j=1}^i\delta_{j,v}$; note that $f^\#\prod_{j=1}^i\delta_{j,v}(\varpi_v)=f^\#\psi(U_{i,v})$.  Moreover, by the construction of $(\varphi,\Gamma)$-modules, there is a finite cover $\{V_j':=\Spa R_j'\}$ of $X'$ by affinoids and a finite extension $L_v/F_v$ such that $D_{\rig}^{L_v}(M_v'|_{V_j'})$ is a free $\Lambda_{V_j',\rig,L_V}$-module for each $j$.

		Thus, we are in the following situation: We have a $X'$-locally free $(\varphi,\Gamma)$-module $D$ over $X'$, an element $\alpha\in R$ which is non-vanishing away from $Z$, and a character $\delta\colon L_v^\times\rightarrow \O_{X'}(X'\smallsetminus \widetilde Z)^\times$ with $\delta(\varpi)=\alpha$, such that at a Zariski-dense set of maximal points $x\in X^{\mathrm{cl}}$, $k_x\otimes D$ is strictly trianguline with $\delta$ giving the first step in the triangulation.  After twisting by $\delta|_{\O_{L_v}^\times}^{-1}$ (which is defined by $\kappa$ and therefore makes sense on all of $X'$, we may also assume that $\delta|_{\O_{L_v}^\times}$ is trivial.  We will show that this implies that $\alpha$ is everywhere non-vanishing.

		Consider the complex
		\[      C_\alpha\colon D_{[a,b]}\xrightarrow{\varphi^f-\alpha,\gamma-1} D_{[a,b/p^f]}\oplus D_{[a,b]}\xrightarrow{(\gamma-1)\oplus(\alpha-\varphi^f)} D_{[a,b/p^f]}    \]
		for some sufficiently small $a,b$, where $f_v$ is the inertial degree of $L_v/\Q_p$.
		By Corollary~\ref{cor: finiteness phi^f-alpha}, $C_\alpha$ is a perfect complex, so the function $x\mapsto \dim_{k_x}H^i(C_{\alpha,x})$ is upper semicontinuous.
		At a Zariski-dense set of maximal points $x\in X'$, $H^0(C_{\alpha,x})$ has $k_x$-dimension at least $1$, so this holds for all $x\in X'$.  But if $\alpha$ vanishes at a point $x$, this contradicts the injectivity of $\varphi$ on $(\varphi,\Gamma)$-modules.

	\end{proof}

\subsection{Trianguline points}

In this section, we show that the Galois representations attached to certain characteristic $p$ points of $X_{\G}$ in the closure of the characteristic $0$ eigenvariety are trianguline, partially answering a question of ~\cite{aip2018} and ~\cite{johansson-newton}.

Our setup is similar to the previous section:
\begin{thm}
	Suppose $\mathscr{X}_{\G}^{\mathscr{A}_p^+}$ is an extended eigenvariety such that there is a continuous determinant $D\colon \mathscr{O}(\mathscr{X}_{\G}^{\mathscr{A}_p^+\red})[\Gal_{F,S}]\rightarrow \mathscr{O}(\mathscr{X}_{\G}^{\mathscr{A}_p^+,\red})$ for some number field $F$, and let $\mathscr{X}\hookrightarrow \mathscr{X}_{\G}^{\mathscr{A}_p^+,\red}$ be an irreducible Zariski-closed subspace.  Suppose in addition there is an ordered basis $\{s_{i,v}\}$ of $\Sigma$ with $\prod_{j=1}^is_{j,v}\in\Sigma^+$ for all $i$ such that for a very Zariski dense set of points $X^{\mathrm{cl}}\subset \mathscr{X}$ the Galois representation attached to $D$ is trianguline at all places $v\mid p$, with parameters $\{\underline{\delta_{v}}\}$ induced by $\{s_{i,v}\}$.  If $x\in \mathscr{X}$ is a maximal point whose residue field has positive characteristic, then the Galois representation associated to the restriction $D|_x$ is also trianguline at all places $v\mid p$, with parameters $\{\underline{\delta_{v}}\}$ induced by $\{s_{i,v}\}$.
\end{thm}
\begin{proof}
	Let $U=\Spa R\subset X_{\G}^{\red}$ be an irreducible affinoid pseudorigid subspace containing $x$, with $U\smallsetminus \{p=0\}$ non-empty.  By ~\cite[Theorem 3.8]{wang-erickson}, there is a topologically finite-type cover $f'\colon U'\coloneqq \Spa R'\rightarrow U$ and a Galois representation $\rho'\colon \Gal_{F,S}\rightarrow \GL_n({R'}^\circ)$ such that the determinant associated to $\rho'$ is equal to ${R'}^\circ\otimes_{R^\circ}D$.  By ~\cite[Theorem 1.1]{bellovin2020}, for each place $v\mid p$ of $F$, there is a projective $(\varphi,\Gamma_{F_v})$-module $D_{\rig}(\rho'_v)$ associated to $\rho'_v$.

	By assumption, there is a Zariski-dense set of points $\{x_i\}\subset U\smallsetminus\{p=0\}$ and continuous characters $\delta_{v,j}\colon F_v^\times\rightarrow R^\times$ such that the $(\varphi,\Gamma_{F_v})$-module attached to $D_{x_i}$ is trianguline with parameters $\{\underline\delta_{v}\}$.  Thus, by Corollary~\ref{cor: interpolate triangulation}, there is a further cover $f''\colon U''\rightarrow U$ and there is an open subspace $V\subset U''$ containing $\{p=0\}\subset U''$ such that ${f''}^\ast\rho'|_V$ is trianguline with parameters $\{\delta_{v}\}$.

In particular, ${f''}^\ast\rho'|_{(f'\circ{f''})^{-1}(x)}$ is trianguline. Since $(f'\circ{f''})^{-1}(x)\rightarrow \{x\}$ is faithfully flat, the triangulation descends to a triangulation on $D_{\rig}(\rho_x)$ with the desired parameters.
\end{proof}

This argument applies strictly to positive characteristic points lying in the closure of points where the Galois representation is already known to be trianguline.  In particular, if there are irreducible components supported entirely in positive characteristic, we can say nothing at all.  However, ~\cite[Lemma 4.2.2]{johansson-newton17} implies that the extended Coleman--Mazur eigencurve does not contain any strictly characteristic $p$ components, and so the Galois representations associated to its boundary points are all trianguline at $p$.

\appendix

\section{Complements on pseudorigid spaces}
\label{app: pseudorigid}

Let $E$ be a complete discretely valued field with ring of integers $\mathscr{O}_E$, uniformizer $\varpi_E$, and residue field $k_E$.  We briefly recall the definition of a pseudorigid space over $\mathscr{O}_E$, before discussing pseudorigid generalizations of the maximum modulus principle and the generic fiber constructions of Bosch--L\"utkebohmert~\cite{bosch-lut} and Berthelot~\cite{dejong1995crystalline}.

Pseudoaffinoid algebras (which are the building blocks for pseudorigid spaces) were defined in ~\cite[Definition 4.3]{lourenco}:
\begin{definition}
	Let $R$ be a Tate ring.  We say that $R$ is a \emph{pseudoaffinoid $\mathscr{O}_E$-algebra} if it has a noetherian ring of definition $R_0\subset R$ which is formally of finite type over $\mathscr{O}_E$.  If $X$ is an adic space over $\Spa \mathscr{O}_E$, we say that $X$ is \emph{pseudorigid} if it is locally of the form $\Spa R\coloneqq  \Spa (R,R^\circ)$, where $R$ is a pseudoaffinoid $\mathscr{O}_E$-algebra.
\end{definition}

\begin{example}
	Let $\lambda=\frac n m\in \Q_{>0}$ be a positive rational number with $(n,m)=1$, and set $D_\lambda^\circ\coloneqq \mathscr{O}_E[\![u]\!]\left\langle \frac{\varpi_E^m}{u^n}\right\rangle$ and $D_\lambda\coloneqq D_\lambda^\circ\left[\frac 1 u\right]$.  Then $D_\lambda$ is a pseudoaffinoid algebra.
\end{example}
Every pseudoaffinoid algebra $R$ is a topologically finite type $D_\lambda$-algebra for some sufficiently small $\lambda>0$, by ~\cite[Lemma 4.8]{lourenco}.

\subsection{Maximum modulus principle}\label{app: max mod principle}

In classical rigid analytic geometry, the maximum modulus principle states roughly that every function on an affinoid domain attains its supremum at some closed point.  We wish to give an analogous result for affinoid pseudorigid spaces.  We note that we have chosen to present it using the language of valuations, rather than norms, since there is no longer a natural exponential base, so what we prove might better be called a ``minimum valuation principle''.

Let $E$ and $D_{\lambda}$ be as above, and let $D_{\lambda,r}\coloneqq D_\lambda\left\langle X_1,\ldots,X_r\right\rangle$, which is a pseudoaffinoid algebra corresponding to a closed ball over $D_\lambda$.  Let $D_\lambda^{\circ\circ}$ denote the ideal of topologically nilpotent elements of $D_\lambda^\circ$.  

We begin by defining valuations on $D_{\lambda}$ and on Tate algebras over it.  Each element of $D_\lambda$ may be written uniquely in the form $\sum_{i\in\Z}a_iu^i$ with $a_i\in\mathscr{O}_E$, which permits the following definition:
\begin{definition}
	We define analogues of the Gauss norm on $D_\lambda$ and $D_{\lambda,r}$, via 
	\[	v_{D_\lambda}(\sum_{i\in\Z}a_iu^i)\coloneqq  \inf_i \left\{v_E(a_i)+\frac{i}{\lambda}\right\}	\]
	and
	\[	v_{D_{\lambda,r}}\left(\sum_{j\in \Z^{\oplus r}}a_j\underline X^j\right)\coloneqq \inf_jv_{D_\lambda}(a_j)	\]
	respectively.  

	For any Tate ring $R$ with ring of definition $R_0$ and pseudouniformizer $u\in R_0$, we also define the \emph{spectral semi-valuation}
\[	v_{R,\spc}(f)\coloneqq  -\inf_{\substack{\{(a,b) \in \Z\oplus \N \mid \\ u^af^b\in R_0\}}} \frac{a}{b}       \]
	for $f\in R$.
\end{definition}
Note that $v_{R,\spc}$ depends on a choice of pseudouniformizer, but we suppress this in the notation.

Then it is clear from the definition that $f\in D_{\lambda,r}^{\circ}$ if and only if $v_{D_{\lambda,r}}(f)\geq 0$, and $f$ is topologically nilpotent if and only if $v_{D_{\lambda,r}}(f)>0$.  Moreover, for all $f\in D_{\lambda,r}$, $v_{D_{\lambda,r}}(uf)=\lambda^{-1}+v_{D_{\lambda,r}}(f)$, and $v_{D_{\lambda,r}}(af)=v_E(a)+v_{D_{\lambda,r}}(f)$ for all $a\in \mathscr{O}_E$.  Similarly, for a general pseudoaffinoid algebra $R$ and $f\in R$, $v_{R,\spc}(uf)=1+v_{R,\spc}(f)$ and $v_{R,\spc}(f^k)=k\cdot v_{R,\spc}(f)$ for all integers $k\geq 0$.  

We also observe that $v_{D_\lambda}$ takes values in the discrete subgroup of $\R$ generated by $\lambda^{-1}$ and $v_E(\varpi_E)$.

We give an alternate interpretation of the spectral semi-valuation; I am grateful to the referee for permitting me to include this argument:
\begin{lemma}\label{lemma: spectral berkovich spectrum}
	If $R$ is a Tate ring with ring of definition $R_0$ and pseudouniformizer $u\in R_0$, then for any $f\in R$
	\[	v_{R,\spc}(f) = \inf_{x\in\Spa(R,R^\circ)^{\mathrm{Berk}}} \left\{\frac{\log\lvert f\rvert_x}{\log\lvert u\rvert_x}\right\}	\]
		where the Berkovich spectrum $\Spa(R,R^\circ)^{\mathrm{Berk}}$ denotes the rank-$1$ points of $\Spa(R,R^\circ)$.
	\label{lemma: spectral inf}
\end{lemma}
\begin{proof}
	Let $v'\coloneqq \inf_{x\in\Spa(R,R^\circ)^{\mathrm{Berk}}} \left\{\frac{\log\lvert f\rvert_x}{\log\lvert u\rvert_x}\right\}$.  For any $\varepsilon>0$, we can find $a\in\Z$, $b\in\N$ such that $v_{R,\spc}(f)-\varepsilon <-\frac{a}{b}$ and $u^af^b\in R_0$.  Then for any $x\in \Spa(R,R^\circ)$, $\lvert u^af^b\rvert_x= \lvert u\rvert_x^a\lvert f\rvert_x^b$ and moreover, $\lvert u^af^b\rvert_x\leq 1$ (since $R_0\subset R^\circ$).  Hence $\log\lvert u^af^b\rvert_x= a\log\lvert u\rvert_x + b\log \lvert f\rvert_x \leq 0$ (for any choice of logarithm base), so
	\[	\frac{\log\lvert f\rvert_x}{\log\lvert u\rvert_x}\geq -\frac a b> v_{R,\spc}(f) - \varepsilon	\]
	Since this inequality holds for any choice of $\varepsilon>0$, we have $v'\geq v_{R,\spc}(f)$.

	On the other hand, for any $\varepsilon>0$, we can find $a\in\Z$ and $b\in\N$ so that 
	\[	v'-\varepsilon < -\frac a b < v'	\]
	Then for any $x\in \Spa(R,R^\circ)^{\mathrm{Berk}}$, we have $\frac{\log\lvert f\rvert_x}{\log \lvert u\rvert_x}>-\frac a b$, so $b\log\lvert f\rvert_x< -a\log\lvert u\rvert_x$, and hence $\lvert u^af^b\rvert_x< 1$.  If $y\in \Spa(R,R^\circ)$ has associated rank-$1$ point $x$, then $\lvert u^af^b\rvert_x< 1$ implies $\lvert u^af^b\rvert_y< 1$.
	Hence, in fact, $\lvert u^af^b\rvert_x<1$ for all $x\in\Spa(R,R^\circ)$, and by~\cite[Lemma 3.3(i)]{huber1993continuous}, this implies that $u^af^b$ is power-bounded.  This means there is some $c\in\Z$ such that $u^c(u^af^b)^n\in R_0$ for all $n\in \N$, so 
	\[	v_{R,\spc}(f)\geq -\frac a b - \frac{c}{nb}	\]
	for all $n$.  Since this holds for all $n\in\N$ and all $\varepsilon >0$, we see that $v_{R,\spc}(f)\geq v'$, as desired.
\end{proof}

\begin{lemma}\label{lemma: pos spectral val top nilp}
	If $f\in R$ and $v_{R,\spc}(f)>0$, then $f$ is topologically nilpotent.
\end{lemma}
\begin{proof}
	By assumption, there is some $a\in \Z_{<0}$ and some $b\in \N$ such that $u^af^b\in R_0$, so $f^b\in u^{-a}R_0$ and $f$ is topologically nilpotent.
\end{proof}

\begin{cor}\label{cor: vdlambda vsp both zero}
	If $f\in D_{\lambda,r}$ and $v_{D_{\lambda,r}}(f)=0$, then $v_{D_{\lambda,r},\spc}(f)=0$, as well, where $v_{D_{\lambda,r},\spc}$ is computed with respect to the pseudouniformizer $u\in D_\lambda$.
\end{cor}
\begin{proof}
	The hypothesis implies that $f\in D_{\lambda,r}^\circ\smallsetminus D_{\lambda,r}^{\circ\circ}$, so $v_{D_{\lambda,r},\spc}(f)\geq 0$ by definition.  If we had $v_{D_{\lambda,r},\spc}(f)> 0$, then Lemma~\ref{lemma: pos spectral val top nilp} would imply that $f$ is topologically nilpotent.  Since this is impossible, we must have $v_{D_{\lambda,r},\spc}(f)=0$.
\end{proof}

\begin{lemma}
	If $f\in D_{\lambda,r}$, then $v_{D_{\lambda,r},\spc}(af)=\lambda v_E(a)+v_{D_{\lambda,r},\spc}(f)$ for all $a\in\mathscr{O}_E$, where $v_{D_{\lambda,r},\spc}$ is computed with respect to the pseudouniformizer $u\in D_\lambda$.
	\label{lemma: pi_E multiplicative}
\end{lemma}
\begin{proof}
	We may assume that $f\in D_{\lambda,r}^\circ$ and $a=\varpi_E$.  Then 
	\[	v_{D_{\lambda,r},\spc}(\varpi_E f)\geq v_{D_{\lambda,r},\spc}(\varpi_E)+v_{D_{\lambda,r},\spc}(f)\geq \frac{n}{m}+v_{D_{\lambda,r},\spc}(f)=\lambda+v_{D_{\lambda,r},\spc}(f)	\]
since $u^{-n}\varpi_E^m\in D_{\lambda,r}^\circ$.  

On the other hand, we need to show that if $u^a(u^{-n}\varpi_E^mf^m)^b\in D_{\lambda,r}^\circ$, then $u^af^{bm}\in D_{\lambda,r}^\circ$.  Writing $D_{\lambda,r}^{\circ}\cong \mathscr{O}_E[\![u]\!]\left\langle X,X_1,\ldots,X_r\right\rangle/(u^nX-\varpi_E^m)$, it suffices to show that if $f'\in D_{\lambda,r}$ satisfies $Xf'\in D_{\lambda,r}^\circ$, then $f'\in D_{\lambda,r}^\circ$.  We may write $f'=u^{-N}f''$ for some $f''\in D_{\lambda,r}^\circ\smallsetminus uD_{\lambda,r}^\circ$ and some integer $N\geq 0$.  Then we see that $Xf''= u^N(Xf')$, which is an element of $u^ND_{\lambda,r}^\circ$; in particular, if $N\geq 1$ then $Xf''\equiv 0$ in $D_{\lambda,r}^\circ/u\cong (\mathscr{O}_E/\varpi_E^m)[X,X_1,\ldots,X_r]$.  But $X$ is not a zero-divisor in this ring, so $f''\in uD_{\lambda,r}^\circ$, contradicting our assumption.  Therefore $N=0$ and $f'\in D_{\lambda,r}^\circ$, as desired.
\end{proof}

\begin{cor}\label{cor: vdlambda vsp scaling}
	For $f\in D_{\lambda,r}$, $v_{D_{\lambda,r}}(f)=\lambda^{-1}v_{D_{\lambda,r},\spc}(f)$, where $v_{D_{\lambda,r},\spc}$ is computed with respect to the pseudouniformizer $u\in D_\lambda$.
\end{cor}
\begin{proof}
	There is some $a\in E$ and some $k\in\Z$ such that $v_{D_{\lambda,r}}(au^kf)=0$; by Corollary~\ref{cor: vdlambda vsp both zero}, we also have $v_{D_{\lambda,r},\spc}(au^kf)=0$.  Then
	\[	v_{D_{\lambda,r}}(f) = -v_E(a) - \frac{k}{\lambda}	\]
	and
	\[	v_{D_{\lambda_r},\spc}(f) = -\lambda v_E(a) - k	\]
	by Lemma~\ref{lemma: pi_E multiplicative}, and the result follows.
\end{proof}

\begin{cor}
	If $f\in D_{\lambda,r}$, then $f\in D_{\lambda,r}^\circ$ if and only if $v_{D_{\lambda,r},\spc}(f)\geq 0$, where $v_{D_{\lambda,r},\spc}$ is computed with respect to the pseudouniformizer $u\in D_\lambda$.
\end{cor}
\begin{proof}
	This follows from the same statement for $v_{D_{\lambda,r}}(f)$ via Corollary~\ref{cor: vdlambda vsp scaling}.
\end{proof}

The maximal points of $\Spa D_\lambda$ consist of the points of a classical half-open annulus (which is not quasi-compact), together with a positive characteristic ``limit point'' $\Spa (D_\lambda/\varpi_E)=\Spa k_E(\!(u)\!)$.  Similarly, the maximal points of $\Spa D_{\lambda,r}$ consist of the product of a closed $r$-dimensional unit ball with a classical half-open annulus and a closed $r$-dimensional unit ball over $\Spa k_E(\!(u)\!)$.  

The advantage of working with $v_{R,\spc}$ for a pseudoaffinoid algebra $R$ is that it makes sense on the residue fields of maximal points, letting us compare $v_{R,\spc}(f)$ and $v_{R/\mathfrak{m}_x,\spc}(f(x))$ for $f\in R$ and $x$ a maximal point of $\Spa R$ (here we use the image of $u$ in $R$ to compute the spectral semi-valuation on residue fields).  It was demonstrated in the proof of ~\cite[Lemma 2.2.5]{johansson-newton17} that if $\mathfrak{m}\subset R$ is a maximal ideal, then $\Spa (R/\mathfrak{m})$ is a singleton.  Then for a maximal point $x\in\Spa R$, the corresponding equivalence class of valuations contains the composition of the specialization map $R\rightarrow R/\mathfrak{m}_x$ with $v_{R/\mathfrak{m}_x,\spc}(f(x))$.

Then we have the following analogue of ~\cite[Lemma 3.8.2/1]{bgr}:
\begin{lemma}
	Let $R$ be a pseudoaffinoid algebra over $\mathscr{O}_E$ and let $x\in \MaxSpec R$ be a maximal point, corresponding to a maximal ideal $\mathfrak{m}_x\subset R$.  Then for any $f\in R$,
	\[	v_{R/\mathfrak{m}_x,\spc}(\overline f)\geq v_{R,\spc}(f)	\]
	\label{lemma: spectral semi-val on points}
\end{lemma}
\begin{proof}
	The quotient map $R\rightarrow R/\mathfrak{m}_x$ carries $R_0$ to the ring of definition of $R/\mathfrak{m}_x$.  Thus, if $u^af^b\in R_0$ then $\overline{u}^a\overline{f}^b\in (R/\mathfrak{m}_x)^\circ$.
\end{proof}

Thus, the spectral semi-valuation is a lower bound for the residual spectral semi-valuations.  In fact, it is the minimum.  Analogously to the classical setting, we first prove this for $R=D_{\lambda,r}$ before deducing the general result.
\begin{lemma}\label{lemma: max mod d-lambda-r}
	For any $f\in D_{\lambda,r}$, the function $x\mapsto v_{D_{\lambda,r}/\mathfrak{m}_x,\spc}(f)$ on $\MaxSpec D_{\lambda_r}$ attains its infimum, and its minimum is equal to $v_{D_{\lambda,r},\spc}(f)$.
\end{lemma}
\begin{proof}
	Fix $f\in D_{\lambda,r}$.  After scaling by an element of the form $au^k$ for $a\in\mathscr{O}_E$, we may assume that $v_{D_{\lambda,r}}(f)=v_{D_{\lambda,r},\spc}(f)=0$, so $f\in D_{\lambda,r}^\circ$ and we need to find a maximal point $x\in \Spa D_{\lambda,r}$ such that $v_{D_{\lambda,r},\spc}(f(x))=0$.  Let $\overline f=\sum_ja_jX^j$ denote the image of $f$ in $D_{\lambda,r}^\circ/D_{\lambda,r}^{\circ\circ}\cong k_E[X,X_1,\ldots,X_r]$.  Since $\overline f\neq 0$ by assumption, there is some closed point $(\overline x,\overline x_1,\ldots,\overline x_r)\in \overline k_E^{r+1}$ such that $\overline f(\overline x,\overline x_1,\ldots, \overline x_r)\neq 0$.  Then for any maximal point $x\in \Spa D_{\lambda,r}$ whose kernel $\mathfrak{m}_x$ reduces to $(\overline x,\overline x_1,\ldots,\overline x_r)$, $v_{D_{\lambda,r},\spc}(f(x))=0$, as desired.
\end{proof}

In order to deduce the same result for more general pseudoaffinoid algebras, we use the Noether normalization result of ~\cite[Proposition 4.14]{lourenco}:
\begin{prop}\label{prop: max mod pseudoaff}
	Let $R$ be an $\mathscr{O}_E$-flat pseudoaffinoid algebra such that $\varpi_E\notin R^\times$ and let $f\in R$.  Then $\inf_{x\in\MaxSpec R} v_{R/\mathfrak{m},\spc}(f)=\min_{x\in\MaxSpec R} v_{R/\mathfrak{m},\spc}(f)=v_{R,\spc}(f)$.
\end{prop}
\begin{proof}
If $\mathfrak{p}_1,\ldots,\mathfrak{p}_s$ are the minimal prime ideals of $R$, we claim that $v_{R,\spc}(f) = \min_j v_{R/\mathfrak{p}_j,\spc}(f)$, so it suffices to prove the result for each $R/\mathfrak{p}_j$.  Indeed, one can consider the natural homomorphism $R\rightarrow\prod_jR/\mathfrak{p}_j$, whose kernel is the nilradical of $R$.  Alternatively, we can write $\Spa(R,R^\circ)^{\mathrm{Berk}} = \cup_j\Spa\left(R/\mathfrak{p}_j,(R/\mathfrak{p}_j)^\circ\right)^{\mathrm{Berk}}$ and apply Lemma~\ref{lemma: spectral berkovich spectrum}.

	Thus, we may assume that $R$ is an integral domain.  The algebra $R/\varpi_E$ is a $k_E(\!(u)\!)$-affinoid algebra, so Noether normalization for affinoid algebras provides us with a finite injective map $k_E(\!(u)\!)\left\langle X_1,\ldots,X_r\right\rangle\rightarrow R/\varpi_E$ for some $r\geq 0$. Then ~\cite[Proposition 4.14]{lourenco} implies that it lifts to a finite injective map $D_{n,r}\rightarrow R\left\langle\frac{\varpi_E}{u^n}\right\rangle$ for some sufficiently large integer $n$.  Then we may argue as in the proof of ~\cite[Proposition 3.8.1/7]{bgr} to see that
	\begin{equation*}
	\begin{split}
\inf_{x\in\MaxSpec R\left\langle\frac{\varpi_E}{u^n}\right\rangle} v_{R/\mathfrak{m},\spc}(f) &= \inf_{y\in \MaxSpec D_{n,r}} \min_{x\in \MaxSpec R\left\langle\frac{\varpi_E}{u^n}\right\rangle\otimes D_{n,r}/\mathfrak{m}_y} v_{R/\mathfrak{m},\spc}(f) 	\\
&= \inf_{y\in \MaxSpec D_{n,r}} \min_j \frac{1}{d-j}v_{D_{n,r}/\mathfrak{m}_y,\spc}(b_j(y))	\\
&= \min_j v_{D_{n,r},\spc}(b_j)
\end{split}
\end{equation*}
where $Y^d+b_{d-1}Y^{d-1}+\ldots+b_0$ is the minimal polynomial for $f$ over $D_{n,r}$.  Since $v_{D_{n,r},\spc}(b_j)$ attains its infimum on $\MaxSpec D_{n,r}$ by Lemma~\ref{lemma: max mod d-lambda-r} and the fibers of $\MaxSpec R\left\langle\frac{\varpi_E}{u^n}\right\rangle$ over $\MaxSpec D_{n,r}$ are finite, $v_{R/\mathfrak{m},\spc}(f)$ also attains its infimum.

Since $v_{R,\spc}$ is ``power-additive'' for all $R$, ~\cite[Proposition 3.1.2/1]{bgr} implies in addition that $v_{R\left\langle \frac{\varpi_E}{u^n}\right\rangle,\spc}(f)\geq \min_j \frac{1}{d-j}v_{D_{n,r},\spc}(b_j)$.  Since the right-side is equal to $\inf_{x\in\MaxSpec R\left\langle\frac{\varpi_E}{u^n}\right\rangle} v_{R/\mathfrak{m},\spc}(f)$, Lemma~\ref{lemma: spectral semi-val on points} implies that $\min_{x\in\MaxSpec R\left\langle \frac{\varpi_E}{u^n}\right\rangle} v_{R/\mathfrak{m},\spc}(f)=v_{R\left\langle \frac{\varpi_E}{u^n}\right\rangle,\spc}(f)$.

To conclude, we use the result for classical affinoid algebras on $\Spa R\left\langle \frac{u^n}{\varpi_E}\right\rangle$.  Since $\Spa R\left\langle\frac{\varpi_E}{u^n}\right\rangle\cup \Spa R\left\langle \frac{u^n}{\varpi_E}\right\rangle$ is a cover of $\Spa R$, the result follows.
\end{proof}

\subsection{Analytic loci of formal schemes}

Recall that a point of a pre-adic space $\Spa(A,A^+)$ is said to be \emph{analytic} if the kernel of the corresponding valuation is not open.  We can describe the analytic locus in certain $\mathscr{O}_E$-formal schemes as explicit pseudorigid spaces, following ~\cite{dejong1995crystalline}.  Let $\mathfrak{X}=\Spf A$ be a noetherian affine formal scheme over $\mathscr{O}_E$, and let $X=\Spa A$ be to corresponding adic space; let $I\subset A$ be the ideal of topologically nilpotent elements, and assume in addition that $A/I$ is a finitely generated $k_E$-algebra.

If $I=(f_1,\ldots,f_r)$, then $X^{\an}=\cup_iX\left\langle \frac{I}{f_i}\right\rangle$.  In the special case when $A=R_0$ is the ring of definition of a $\mathscr{O}_E$-pseudoaffinoid algebra $R$, we have $X^{\an}=\Spa R$.

As in ~\cite[Proposition 7.1.7]{dejong1995crystalline}, we can give a functor-of-points characterization of $X^{\an}$:
\begin{prop}\label{prop: analytic locus functor of points}
	Let $\mathfrak{X}$ and $X$ be as above, and let $Y\coloneqq \Spa R$ be an affinoid pseudorigid space.  Then
	\begin{equation}
	\label{eqn: formal model functor of points}
	\varinjlim_{\substack{R_0\subset R\\ \text{ ring of definition}}}\Hom_{\mathcal{FS}/\mathscr{O}_E}(\Spf R_0,\mathfrak{X})\xrightarrow{\sim}\Hom(Y,X^{\an})
	\end{equation}
	is an isomorphism.
\end{prop}
\begin{proof}
	Given any two rings of definition of $R$, there is a third which contains both of them.  Moreover, suppose $R_0\subset R$ is a ring of definition, and $g\colon R\rightarrow R'$ is a continuous homomorphism of pseudoaffinoid algebras.  By ~\cite[Proposition 1.10]{huber1993continuous}, $g$ is adic; it therefore carries $R_0\subset R$ to a ring of definition of $R'$.

	Thus, for a fixed $\mathfrak{X}$, we can view $\varinjlim_{R_0\subset R\text{ ring of definition}}\Hom_{\mathcal{FS}/\mathscr{O}_E}(\Spf R_0,\mathfrak{X})$ as a covariant functor evaluated on $R$, and equation ~\ref{eqn: formal model functor of points} as a natural transformation.  We will construct an inverse.  Suppose we have a morphism $Y\rightarrow X^{\an}$; it is induced by a continuous ring homomorphism $g\colon A\rightarrow R$.  Using the description of the analytic locus of $X$, we see that the image of $Y$ is contained in $U\coloneqq \cup_iU_i$, where $U_i\coloneqq X\left\langle \frac{I}{f_i}\right\rangle$, for some finite set $\{f_i\}\subset I$.  

	Let $V_i\subset Y$ denote the rational subset $\Spa R\left\langle\frac{g(I)}{g(f_i)}\right\rangle$ (note that by ~\cite[Proposition 3.8(iii)]{huber1993continuous}, $g(I)\subset R'$ is open, so rational localization makes sense), and let $R_i$ denote its coordinate ring.  Each morphism $V_i\rightarrow U_i$ is induced by a continuous ring homomorphism $A\left\langle \frac{I}{f_i}\right\rangle\rightarrow R_{0,i}$, for some ring of definition $R_{0,i}\subset R_i$.	

Let $R_0\subset R^\circ$ be the equalizer of $\prod_iR_{0,i}\rightrightarrows \prod_{i,j}R_{i,j}^\circ$; we claim that the map $A\rightarrow \Gamma(U,\mathscr{O}_X)\rightarrow R$ factors through $R_0$, and $R_0$ is a ring of definition of $R$.  For the first claim, we consider the diagram
	\[
		\begin{tikzcd}[row sep=scriptsize, column sep=scriptsize]
			0 \arrow{rr} & & \Gamma(U,\mathscr{O}_X^+) \arrow[dashed]{dr}\arrow{rr} & & {\prod_i} \Gamma(U_i,\mathscr{O}_X^+) \arrow[shift left]{rr}\arrow[shift right]{rr}\arrow{dr} & & \prod_{i,j}\Gamma(U_{i,j},\mathscr{O}_X^+) \arrow{dd}	\\
			& 0 \arrow{rr} & & R_0 \arrow[hook]{dl}\arrow{rr} & &  \prod_i R_{0,i} \ar[hook]{dl}\ar[shift left]{dr}\arrow[shift right]{dr} &	\\
			0 \arrow{rr} & & R^\circ \arrow[crossing over,leftarrow]{uu}\arrow{rr} & & {\prod_i} R_i^\circ \arrow[crossing over,leftarrow]{uu}\arrow[shift left]{rr}\arrow[shift right]{rr} & & \prod_{i,j}R_{i,j}^\circ
		\end{tikzcd}
	\]
It is commutative with exact rows, and a diagram chase shows that the dotted arrow exists.

For the second claim, let $u\in R$ be a pseudouniformizer; we check that $R_0[u^{-1}]=R$ and $R_0$ is a bounded subring of $R$.  Given $r\in R$, we may write $\prod_i r_i$ for its image in $\prod_iR_i$.  Since $R_{0,i}\subset R_i$ is a ring of definition, there is some $n\in \N$ such that $u^n\left(\prod_i r_i\right)\in \prod_iR_{0,i}$; by construction, $u^n\left(\prod_i r_i\right)$ is in the kernel of $\prod_iR_{0,i}\rightrightarrows\prod_{i,j}R_{i,j}^\circ$, so it defines the desired element of $R_0$.

To see that $R_0\subset R$ is bounded, we let $R_0'\subset R$ be a ring of definition. It induces rings of definition $R_{0,i}'\subset R_i$; since $R_{0,i}\subset R_i$ is bounded, there is some $n'\in \N$ such that $u^{n'}\prod_iR_{0,i}\subset \prod_iR_{0,i}'$, and a diagram chase shows that $u^{n'}R_0\subset R_0'$.

We have constructed a continuous homomorphism $A\rightarrow R_0$ inducing the morphism $Y\rightarrow X^{\an}$, where $R_0\subset R$ is a ring of definition.  The corresponding morphism $\Spf R_0\rightarrow \Spf A$ is the desired element of the left side of equation~\ref{eqn: formal model functor of points}.  It is straightforward to verify that this defines a natural transformation.
\end{proof}

\subsection{Local structure of pseudorigid spaces}

We give a result on the local structure of pseudorigid spaces, using the theory of formal models developed in ~\cite{bosch-lut} and ~\cite{bosch-lut2}.  Although the authors had in mind applications to classical rigid analytic spaces, they worked in sufficient generality that their results hold in the more general pseudorigid context.  More precisely, they related rigid spaces to admissible formal schemes over a base formal scheme $S$, where $S$ could be the formal spectrum of an arbitrary noetherian adic ring.  Since a pseudoaffinoid algebra is topologically of finite type over some $D_\lambda$ and $D_\lambda^\circ$ is noetherian and complete for the $u$-adic topology, this is sufficient for our purposes.

We give the definition of admissible formal schemes, from ~\cite[\textsection 1]{bosch-lut}:
\begin{definition}
	Let $A$ be a ring with a finitely generated ideal $J\subset A$ such that $A$ is $J$-adically complete and has no $J$-torsion.  An $A$-algebra $R$ is \emph{admissible} if $R\cong A\left\langle X_1,\ldots, X_n\right\rangle/I$ for some finitely generated ideal $I\subset A\left\langle X_1,\ldots, X_n\right\rangle$ and $R$ has no $J$-torsion.

	An affine formal $A$-scheme $\mathfrak{X}\coloneqq \Spf R$ is \emph{admissible} if $R$ is an admissible $A$-algebra.  A quasi-compact formal $A$-scheme is said to be admissible if it has a cover by admissible formal $A$-schemes.
	\label{def: admissible}
\end{definition}
There is an implicit assertion here that admissibility can be checked locally on $\mathfrak{X}$, which follows from ~\cite[Proposition 1.7]{bosch-lut}; we refer the reader there for more details.

\begin{thm}
	Let $f\colon R\rightarrow R'$ be a continuous homomorphism of pseudoaffinoid algebras over $\O_E$, such that for every maximal ideal $\mathfrak{m}\subset R$, the induced maps $R/\mathfrak{m}^n\rightarrow R'/\mathfrak{m}^n$ are isomorphisms.  Then $f$ is an isomorphism.
	\label{thm: local isom aff}
\end{thm}
\begin{proof}
	To begin with, we observe that $f$ induces a bijection $\MaxSpec(R')\rightarrow\MaxSpec(R)$.  Furthermore, by ~\cite[Tag 0523]{stacks-project}, the maps $R_{\mathfrak{m}}\rightarrow R_{\mathfrak{m}}'$ on algebraic localizations are flat.

	By ~\cite[Lemma 4.8]{lourenco} there is some $\lambda\in\Q_{>0}$ such that $R$ is topologically of finite type over $D_\lambda$, so we may choose a ring of definition $R_0\subset R$ which is topologically of finite type over $D_\lambda^\circ$.  Similarly, since $R'$ is topologically of finite type over $R$ by ~\cite[Corollary A.14]{johansson-newton}, we may choose a ring of definition $R_0'\subset R'$ which is topologically of finite type over $R_0$.  Setting $\mathfrak{X}=\Spf R_0$ and $\mathfrak{Y}=\Spf R_0'$, we have a morphism $\mathfrak{Y}\rightarrow \mathfrak{X}$ of admissible formal $\Spf D_\lambda^\circ$-schemes which corresponds to $f$ after inverting $u$.

	By ~\cite[Theorem 5.2]{bosch-lut2}, there are admissible formal blow-ups $\widetilde{\mathfrak{X}}\rightarrow\mathfrak{X}$ and $\widetilde{\mathfrak{Y}}\rightarrow\mathfrak{Y}$ and a flat morphism $\widetilde{\mathfrak{Y}}\rightarrow\widetilde{\mathfrak{X}}$.  By ~\cite[Corollary 5.3]{bosch-lut2}, we may also assume this morphism is quasi-finite.  We claim this morphism is surjective.  Indeed, flat morphisms are open; if it is not surjective, the complement $\mathfrak{Z}\subset\widetilde{\mathfrak{X}}$ of its image is a closed formal subscheme, whose associated pseudorigid generic fiber $Z\subset \Spa R$ contains maximal points which are not in the image of $\Spa R'$, which contradicts our assumptions.

	Now we consider the morphism $\widetilde{\mathfrak{Y}}_0\rightarrow\widetilde{\mathfrak{X}}_0$ between the mod $u$ fibers.  This is a separated flat quasi-finite morphism between finite-type $\F_p$-schemes, and we claim it is actually an isomorphism.  For this it suffices to show that the fibral rank is constant and equal to $1$ (since this implies it is finite of rank $1$, by ~\cite[Lemma II.1.19]{deligne-rapoport}).  Moreover, the locus in $\widetilde{\mathfrak{X}}_0$ of points with fibral rank-$d$ is constructible by ~\cite[Lemma 9.8.8]{ega-iv-3}, and constructible sets contain closed points, so it suffices to check that all closed points have rank-$1$ fibers.

	But if $x_0\in \widetilde{\mathfrak{X}}_0$ is a closed point, there is some local integral domain $A$ of dimension $1$ and a morphism $\Spf A\rightarrow \widetilde{\mathfrak{X}}$ whose reduction modulo $u$ is $x_0$, by ~\cite[Proposition 3.5]{bosch-lut}.  The base change $\widetilde{\mathfrak{Y}}_x\rightarrow \Spf A$ is again flat and quasi-finite; since $A$ is a local ring, the rank of its fiber over $A\left[\frac 1 u\right]$ is equal to its fiber over the residue field.  But by assumption, the fiber over $A\left[\frac 1 u\right]$ has rank $1$, so the fiber $\widetilde{\mathfrak{Y}}_{x_0}$ has rank $1$, as desired.

	Now we have a morphism of sheaves of topological rings $\O_{\widetilde{\mathfrak{X}}}\rightarrow \O_{\widetilde{\mathfrak{Y}}}$ which is surjective modulo $u$.  By ~\cite[Tag 07RC(11)]{stacks-project}, it is surjective modulo all powers of $u$.  Since $\O_{\widetilde{\mathfrak{X}}}$ and $\O_{\widetilde{\mathfrak{Y}}}$ are $u$-adically complete, it is surjective.

	Suppose $\O_{\widetilde{\mathfrak{X}}}\rightarrow \O_{\widetilde{\mathfrak{Y}}}$ has kernel sheaf $\mathscr{J}$.  Since $\O_{\widetilde{\mathfrak{X}}_0}\rightarrow \O_{\widetilde{\mathfrak{Y}}_0}$ is an isomorphism and $\widetilde{\mathfrak{Y}}$ is $\widetilde{\mathfrak{X}}$-flat, $\mathscr{J}/u=0$ and by ~\cite[Tag 07RC(11)]{stacks-project} $\mathscr{J}=0$.

	Thus, $\widetilde{\mathfrak{Y}}\rightarrow\widetilde{\mathfrak{X}}$ is an isomorphism; inverting $u$, we see that $R\rightarrow R'$ is an isomorphism, as desired.
\end{proof}
	
\section{Pseudorigid determinants}\label{appendix: determinants}

We need to extend some of the results of~\cite{chenevier2014p} on moduli spaces of Galois determinants from the rigid analytic setting to the pseudorigid setting.  Recall that for any topological group $G$ and $d\in \N$, ~\cite{chenevier2014p} defines functors $\widetilde E_d\colon \mathcal{FS}/\Z_p\rightarrow \underline{\mathrm{Set}}$ and $\widetilde E_{d,z}\colon \mathcal{FS}/\Z_p\rightarrow \underline{\mathrm{Set}}$ on the category of formal schemes over $\Z_p$, where $z$ is a $d$-dimensional determinant $G\rightarrow k$ for some finite field $k$.  More precisely,
\[	\widetilde E_d(\mathfrak{X}) \coloneqq  \{\text{continuous determinants }\mathscr{O}(\mathfrak{X})[G]\rightarrow \mathscr{O}(\mathfrak{X})\text{ of dimension }d\}	\]
and $\widetilde E_{d,z}(\mathfrak{X}) \subset \widetilde E_d(\mathfrak{X})$ is the subset of continuous determinants which are residually constant and equal to $z$.

Suppose that $G$ is a topological group satisfying the following property:

For any open subgroup $H\subset G$, there are only finitely many continuous group homomorphisms $H\rightarrow \Z/p\Z$.

Under this condition (which is satisfied by absolute Galois groups of characteristic $0$ local fields, and by groups $\Gal_{F,S}$, where $F$ is a number field and $S$ is a finite set of places of $F$), ~\cite[Corollary 3.14]{chenevier2014p} implies that $\widetilde E_d$ and $\widetilde E_{d,z}$ are representable.  Moreover, every continuous determinant is residually locally constant, and so $\widetilde E_d=\coprod_z \widetilde E_{d,z}$.

We may define an analogous functor $\widetilde E_d^{\an}$ on the category of pseudorigid spaces, and we wish to prove the following:
\begin{thm}\label{thm: analytic moduli of determinants}
	The functor $\widetilde E_d^{\an}$ is representable by a pseudorigid space $X_d$, and $X_d$ is canonically isomorphic to the analytic locus of $\widetilde E_d$.  The functor $\widetilde E_{d,z}^{\an}$ is representable by a pseudorigid space $X_{d,z}$, and $X_{d,z}$ is canonically isomorphic to the analytic locus of $\widetilde E_{d,z}$.  Moreover, $\widetilde E_d^{\an}$ is the disjoint union of the $\widetilde E_{d,z}^{\an}$.
\end{thm}
\begin{remark}
	This is a direct analogue of ~\cite[Theorem 3.17]{chenevier2014p}, and the proof is virtually identical.  However, we sketch it here for the convenience of the reader.
\end{remark}
If $R$ is a pseudoaffinoid algebra, it contains a noetherian ring of definition $R_0\subset R$, and we have $R^\circ=\varinjlim R_0$, where $R^\circ\subset R$ is the subring of power-bounded elements of $R$ and the colimit is taken over all rings of definition of $R$ .  We have an injective map
\[	\iota\colon  \varinjlim_{\substack{R_0\subset R\\ \text{ ring of definition}}} \widetilde E(\Spf R_0)\rightarrow \widetilde E^{\an}(R)	\]
Exactly as in ~\cite[Lemma 3.15]{chenevier2014p}, we have the following:
\begin{lemma}\label{lemma: analytic points of E}
	Let $R$ be a pseudoaffinoid algebra, and let $D\in \widetilde E_d^{\an}(R)$.  Then
	\begin{enumerate}
		\item	For all $g\in G$, the coefficients of $D(1+gt)\in R[t]$ lie in $R^\circ$.
		\item	The map $\iota\colon \varinjlim_{R_0\subset R} \widetilde E(\Spf R_0)\rightarrow \widetilde E^{\an}(R)$ is bijective.
		\item	If $R$ is reduced, then $\widetilde E_d(\Spf R^\circ)=\widetilde E_d^{\an}(R)$.
	\end{enumerate}
\end{lemma}
In particular, if $L$ is the residue field at a maximal point of $\Spa R$ and $\mathscr{O}_L$ is its ring of integers, $\widetilde E_d(\Spf \mathscr{O}_L)=\widetilde E_d^{\an}(L)$.  Every $L$-valued point of $\Spa R$ therefore defines a map
\[	\widetilde E_d(\Spa R)\rightarrow \widetilde E_d(L)=\widetilde E_d(\Spf \mathscr{O}_L)\rightarrow \widetilde E_d(k_L)	\]
where $k_L$ is the residue field of $\mathscr{O}_L$.

Thus, we may talk about residual determinants of determinants $R[G]\rightarrow R$, and define $\widetilde E_{d,z}^{\an}$ for any continuous determinant $z$ valued in a finite field.

Now the proof of Theorem~\ref{thm: analytic moduli of determinants} follows by combining Proposition~\ref{prop: analytic locus functor of points} and Lemma ~\ref{lemma: analytic points of E}.

For the convenience of the reader, we record the following analogue of ~\cite[Lemma 7.8.11]{bellaiche-chenevier}; its proof carries over verbatim to the setting of pseudorigid families of determinants.
\begin{lemma}
	Let $D\colon G\rightarrow \mathscr{O}(X)$ be a continuous $d$-dimension determinant of a topological group on a reduced pseudorigid space $X$.  Let $U\subset X$ be an open affinoid.
	\begin{enumerate}
		\item	There is a normal affinoid $Y$, a finite surjective map $g\colon Y\rightarrow U$, and a finite-type torsion-free $\mathscr{O}(Y)$-module $M(Y)$ of generic rank $d$ equipped with a continuous representation $\rho_Y\colon G\rightarrow \GL_{\mathscr{O}(Y)}$, whose determinant at generic points of $Y$ agrees with $g^\ast D$.

			Moreover, $\rho_Y$ is generically semisimple and the sum of absolutely irreducible representations.  For $y$ in a dense Zariski-open subset $Y'\subset Y$, $M(Y)_y$ is free of rank-$d$ over $\mathscr{O}_y$, and $M(Y)_y\otimes\overline{k(y)}$ is semisimple and isomorphic to $\overline\rho_{g(y)}$.
		\item	There is a blow-up $g'\colon \mathcal{Y} \rightarrow Y$ of a closed subset of $Y\smallsetminus Y'$ such that the strict transform $\mathcal{M}_{\mathcal{Y}}$ of the coherent sheaf on $Y$ associated to $M(Y)$ is a locally free $\mathcal{O}_{\mathcal{Y}}$-module of rank $d$.  That sheaf $\mathcal{M}_{\mathcal{Y}}$ is equipped with a continuous $\mathcal{O}_{\mathcal{Y}}$-representation of $G$ with determinant $(g'g)^\sharp\circ D$, and for all $y\in Y$, $(\mathcal{M}_{Y,y}\otimes \overline{k(y)})^{\mathrm{ss}}$ is isomorphic to $\overline\rho_{g'g(y)}$.
	\end{enumerate}
	\label{lemma: det rep after blowup}
\end{lemma}

\end{document}